\documentclass{article}

\usepackage{epsfig}
\usepackage{amsmath,amsthm,amsfonts,amssymb,euscript}
\usepackage{graphicx}
\usepackage{pgf}
\usepackage{caption}
\usepackage{subcaption}
\usepackage{tikz,tkz-tab}
\usepackage{array,multirow,makecell}
\usepackage{tabularx}
\usepackage{hyperref}
\usepackage{makeidx}
\usepackage{geometry}
\usepackage{authblk}
\def\XXint#1#2#3{{\setbox0=\hbox{$#1{#2#3}{\int}$}
    \vcenter{\hbox{$#2#3$}}\kern-.5\wd0}}

\def\longrightharpoonup{\DOTSB\relbar\joinrel\rightharpoonup}

\def\RR{\mathbb R}
\def\ZZ{\mathbb Z}

\def\11{\mathbf{1}}
\begin{document}
\numberwithin{equation}{section}
\newtheorem{theoreme}{Theorem}[section]
\newtheorem{proposition}[theoreme]{Proposition}
\newtheorem{remarque}[theoreme]{Remark}
\newtheorem{lemme}[theoreme]{Lemma}
\newtheorem{corollaire}[theoreme]{Corollary}
\newtheorem{definition}[theoreme]{Definition}

\title{Homogenization of the Poisson equation in a non-periodically perforated domain}
\author[1]{X. Blanc}
\author[1]{S. Wolf}
\affil[1]{{\footnotesize Universit\'e de Paris, Sorbonne Universit\'e, CNRS, Laboratoire Jacques-Louis Lions, F-75013 Paris}}

\maketitle
\begin{abstract}
We study the Poisson equation in a perforated domain with homogeneous Dirichlet boundary conditions. The size of the
perforations is denoted by $\varepsilon>0$, and is proportional to the distance between neighbouring
perforations. In the periodic case, the homogenized problem (obtained in the limit $\varepsilon\to 0$) is well understood (see \cite{lions-1980}). We extend
these results to a non-periodic case which is defined as a localized deformation of the periodic setting. We propose
geometric assumptions that make precise this setting, and we prove results which extend those of the periodic
case: existence of a corrector, convergence to the homogenized problem, and two-scale expansion. 
\end{abstract}

\tableofcontents

\newpage

\section{Introduction}
\label{Introduction}

In this article, we study the following problem: 
\begin{equation}
\label{poisson}
\begin{cases}
\begin{aligned}
- \Delta u_{\varepsilon} & = f \ \ \text{in} \ \ \Omega_{\varepsilon} \\
u_{\varepsilon} & = 0 \ \ \text{on} \ \ \partial \Omega_{\varepsilon},
\end{aligned}
\end{cases}
\end{equation}
where $f$ is a given smooth, compactly-supported function (this assumption may be relaxed, as we will see below in
Remarks~\ref{rk1} and \ref{rk3}),
and $\Omega_\varepsilon$ is a perforated domain that we make precise in the following. Our aim is to study the
asymptotic behaviour of $u_\varepsilon$ as $\varepsilon\to 0$, deriving a two-scale expansion and proving convergence
estimates. In \cite{lions-1980}, these results were obtained in the periodic case (that is, if the perforations are a
periodic array of period $\varepsilon$). Here, we adapt this work to a non-periodic setting. Using Assumptions \textbf{(A1)}
and \textbf{(A2)} below, which are inspired from the setting developed in \cite{BLLMilan,BLLcpde,BLLfutur1}, we first prove
the existence of a corrector (Theorem~\ref{theocor} below). While this result is trivial in the periodic case, it
is not in the present setting. Then, we prove the convergence result stated in Theorem~\ref{theo}, which is a generalization of
\cite[Theorem 3.1]{lions-1980} to the present setting. We also prove such a convergence in $L^\infty$ norm
(Theorem~\ref{theoinfini} below), a result which was not proved in \cite{lions-1980}. The crucial point in order to
prove such results is a Poincar\'e inequality with an explicit scaling in $\varepsilon$, for functions vanishing in
the perforations, as in the periodic case (see Lemma~\ref{Poincmic} below in the periodic case, and
Theorem~\ref{Poincmic2} in the non-periodic case).

\medskip

To our knowledge, the first contribution on the homogenization of elliptic problems in perforated domains is \cite{cioranescu-sjp}.
The setting is periodic, the equation is elliptic in divergence form, and the Dirichlet condition on the boundary
of the holes is not $0$. This implies that the limit is not trivial, in contrast to \cite{lions-1980}, where, as we
will see below, $u_\varepsilon (x)\approx \varepsilon^2 f(x)w(x/\varepsilon)$, for some periodic function $w$. The case of
Neumann boundary conditions was studied in \cite{cioranescu-murat-1,cioranescu-murat-2} and  \cite{cioranescu-donato-1988}, where the geometry is
periodic, but the holes are assumed to be asymptotically small compared to the period. In this case, an important tool to study
the problem is the so-called extension operator, which is studied in details in \cite{acerbi}. In
\cite{damlamian-donato-2002}, sufficient conditions on periodic holes are given which allow for homogenization. In
\cite{cioranescu-donato-zaki-cras,cioranescu-donato-zaki}, the case of Robin boundary conditions is addressed, with
the help of the periodic unfolding method
\cite{cioranescu-damlamian-griso,cioranescu-damlamian-donato-griso-zaki}. The case of eigenvalue problems was
considered in \cite{vanninathan}.

In \cite{briane-damlamian-donato-1998}, a formalization in link with the H-convergence was proposed under general assumptions on the perforations. However, the computations are less explicit than in our setting. 
 A general (non-periodic) perforated domain was also considered in \cite{nguetseng-2004}: this setting requires that, among other assumptions, the same perforation is reproduced in some cells of a periodic grid (but not necessarily all of them).






\medskip

In the following subsection, we recall the results proved in the periodic setting in \cite{lions-1980}. Then, in
Subsection~\ref{sec:non-periodic-case}, we study the case of a locally perturbed periodic geometry. We give
conditions on the perforations (inspired from \cite{BLLMilan,BLLcpde,BLLfutur1}), which imply that, away from the defect, the perforations become periodic, and which allow
to prove convergence results similar to those of the periodic case. In Section~\ref{sec:results}, we give the main
results of the article, together with some remarks and comments. Section~\ref{sec:poincare} is devoted to a
Poincar\'e-type inequality which is crucial in our proof. Finally, Section~\ref{sec:proofs} is devoted to the
proofs of the results stated in Section~\ref{sec:results}.

\subsection{The periodic case}
\label{sec:periodic-case}

We start with some notations. We consider the $d-$dimensional unit cube $Q = ]0,1[^d$ with $d \geq 2$. Let $\mathcal{O}_0^{\mathrm{per}}$ be an open
subset of $Q$ such that $\mathcal{O}_0^{\mathrm{per}} \subset \subset Q$ and $Q \setminus
\overline{\mathcal{O}^{\mathrm{per}}_0}$ is connected. For simplicity, we choose to impose that $\overline{\mathcal{O}_0^{\mathrm{per}}}$ cannot
intersect the boundary of $Q$. Suppose, for elliptic regularity, that $\mathcal{O}_0^{\mathrm{per}}$ is a $C^{1,\gamma}$
domain, for some $0 < \gamma < 1$. 

We set, for all $k \in \mathbb{Z}^d$,

\begin{equation}\label{O_k^per}
\mathcal{O}_k^{\mathrm{per}} = \mathcal{O}^{\mathrm{per}}_0 + k \ \ \mathrm{and}  \ \ \mathcal{O}^{\mathrm{per}} = \bigcup_{k \in \mathbb{Z}^d} \mathcal{O}_k^{\mathrm{per}}. 
\end{equation}

If $k \in \mathbb{Z}^d$, we have $\mathcal{O}_k^{\mathrm{per}} \subset\subset Q_k$ where $Q_k := Q + k$. 





Given $\varepsilon > 0$, denote by $\mathcal{O}_{\varepsilon}^{\mathrm{per}}$ the set of $\varepsilon-$perforations :
\begin{equation}\label{O_epsilon^per}
\mathcal{O}_{\varepsilon}^{\mathrm{per}} = \bigcup_{k \in \mathbb{Z}^d} \varepsilon \mathcal{O}_k^{\mathrm{per}} = \bigcup_{k \in \mathbb{Z}^d} \varepsilon (\mathcal{O}_0^{\mathrm{per}} + k) = \varepsilon \mathcal{O}^{\mathrm{per}}.
\end{equation}

We now define some useful functional spaces :
\begin{equation}
\label{eq:h1per}
H^{1,\mathrm{per}}(Q \setminus \overline{\mathcal{O}_0^{\mathrm{per}}} ):= \left\{ u \in H^{1}_{\mathrm{loc}}(\mathbb{R}^d \setminus \overline{\mathcal{O}^{\mathrm{per}}}) \ \mathrm{s.t.} \ u  \ \mathrm{and} \ \partial_i u \ \mathrm{are} \ Q-\mathrm{periodic} \ \mathrm{for} \ \mathrm{all} \ i \in \{1,...,d\} \right\}
\end{equation}
and 
\begin{equation}
\label{eq:h1pernulle}
H^{1,\mathrm{per}}_0(Q \setminus \overline{\mathcal{O}_0^{\mathrm{per}}} ) := \left\{ u \in H^{1,\mathrm{per}}(Q \setminus \overline{\mathcal{O}_0^{\mathrm{per}}}) \ \mathrm{s.t.} \ u_{|\partial \mathcal{O}_0^{\mathrm{per}}} = 0 \right\}.
\end{equation}
The two spaces defined by \eqref{eq:h1per} and \eqref{eq:h1pernulle} are Hilbert spaces for the norm 
$$\| u \|_{H^{1,\mathrm{per}}(Q \setminus \overline{\mathcal{O}^{\mathrm{per}}_0})} := \left(\int_{Q \setminus \overline{\mathcal{O}^{\mathrm{per}}_0}} |u|^2 + \int_{Q \backslash \overline{\mathcal{O}^{\mathrm{per}}_0}} |\nabla u|^2 \right)^{1/2}.$$
In the sequel, a function of $H^{1,\mathrm{per}}(Q \setminus \overline{\mathcal{O}^{\mathrm{per}}_0})$ or of $H^{1,\mathrm{per}}_0(Q \setminus \overline{\mathcal{O}^{\mathrm{per}}_0})$ will naturally be extended to $\mathbb{R}^d \setminus \overline{\mathcal{O}^{\mathrm{per}}}$ by periodicity.

All along the paper, we will denote the $H^1-$semi-norm on a set $V$ by $|\cdot|_{H^1(V)}$:
$$|u|_{H^1(V)} := \left( \int_V | \nabla u|^2 \right)^{1/2}.$$

\vspace{0.4cm}

Let $\Omega$ be a bounded, open and connected domain of $\mathbb{R}^d$. For $\varepsilon > 0$, denote by $\Omega_{\varepsilon} := \Omega \setminus \overline{\mathcal{O}_{\varepsilon}^{\mathrm{per}}}$. Note that $\Omega_{\varepsilon}$ is open and bounded but may not be connected. 

One has
\begin{equation}
\label{domain}
\Omega_{\varepsilon} = \Omega \cap \left( \bigcup_{k \in \mathbb{Z}^d} \varepsilon(Q_k \setminus \overline{\mathcal{O}^{\mathrm{per}}_k}) \right).
\end{equation}
Figure \ref{figper22} shows the set $\Omega_{\varepsilon}$ for two values $\varepsilon_0$ and $\varepsilon_1$ satisfying $\varepsilon_0 > \varepsilon_1$. The set $\Omega_{\varepsilon}$ is colored in light grey.
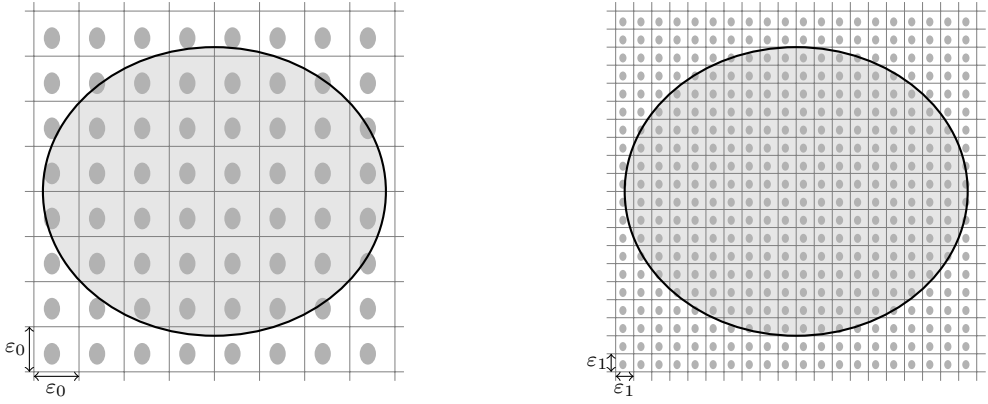
\begin{figure}
\centering
\begin{subfigure}{.5\textwidth}
  \centering
     \begin{tikzpicture}[scale=.12]\footnotesize
 \pgfmathsetmacro{\xone}{-21}
 \pgfmathsetmacro{\xtwo}{21}
 \pgfmathsetmacro{\yone}{-21}
 \pgfmathsetmacro{\ytwo}{21}
 
\fill[gray!20] (0,0) ellipse (19cm and 16cm);

\begin{scope}<+->;
  \draw[step=5cm,gray,very thin] (\xone,\yone) grid (\xtwo,\ytwo);
\end{scope}

\foreach \x in {-20,-15,-10,-5,0,5,10,15}
    \foreach \y in {-20,-15,-10,-5,0,5,10,15}
        \fill[gray!60] (\x+2,\y+2) ellipse (0.9cm and 1.2cm);

\draw[<->] (-20,-20.5) -- (-15,-20.5);
\draw (-17.5,-22) node[]{$\varepsilon_0$};
\draw[<->] (-20.5,-20) -- (-20.5,-15);
\draw (-22,-17.5) node[]{$\varepsilon_0$};

\draw[black,thick] (0,0) ellipse (19cm and 16cm);

\end{tikzpicture}

\end{subfigure}%
\begin{subfigure}{.5\textwidth}
  \centering
\begin{tikzpicture}[scale=.12]\footnotesize
 \pgfmathsetmacro{\xone}{-21}
 \pgfmathsetmacro{\xtwo}{21}
 \pgfmathsetmacro{\yone}{-21}
 \pgfmathsetmacro{\ytwo}{21}
 
\fill[gray!20] (0,0) ellipse (19cm and 16cm);

\begin{scope}<+->;
  \draw[step=2cm,gray,very thin] (\xone,\yone) grid (\xtwo,\ytwo);
\end{scope}

\foreach \x in {-20,-18,-16,-14,-12,-10,-8,-6,-4,-2,0,2,4,6,8,10,12,14,16,18}
    \foreach \y in {-20,-18,-16,-14,-12,-10,-8,-6,-4,-2,0,2,4,6,8,10,12,14,16,18}
        \fill[gray!60] (\x+0.8,\y+0.8) ellipse (0.4cm and 0.5cm);

\draw[<->] (-20,-20.5) -- (-18,-20.5);
\draw (-19,-22) node[]{$\varepsilon_1$};
\draw[<->] (-20.5,-20) -- (-20.5,-18);
\draw (-22,-19) node[]{$\varepsilon_1$};

\draw[black,thick] (0,0) ellipse (19cm and 16cm);
\end{tikzpicture}
\label{figper1/20}
\end{subfigure}

\caption{The periodic set for two choices of $\varepsilon$, $\varepsilon_0$ and $\varepsilon_1 = \varepsilon_0/2.5$.}
\label{figper22}
\end{figure}
We are interested in the Poisson problem~\eqref{poisson}. As we already mentionned, the source term $f$ is
supposed, as in \cite{lions-1980}, smooth and compactly supported in $\Omega$. In fact, (see Remark 3.3 of
\cite{lions-1980}), it is sufficient to suppose that $f \in C^{m-2}(\overline{\Omega})$ and that $D^p f_{|\partial
  \Omega} = 0$ for all $|p| \leq m-2$, where $m$ is the order of the two scale expansion of $u_{\varepsilon}$. As
pointed out in \cite{lebris-legoll-lozinski-2014}, the assumptions on $f$ can be weakened further (see
Remark~\ref{rk1} below).

By a simple application of the Lax-Milgram Lemma, we have existence and uniqueness of a solution $u_\varepsilon$ to \eqref{poisson}.
In order to study the dependance of $u_\varepsilon$ on $\varepsilon$, we will need the following Lemma which is a Poincar\'e-type inequality
in $H^1_0(\Omega_{\varepsilon})$. It is proved in \cite[Lemma 1]{tartar} (see also \cite[Proposition
3.1]{bourgeat-tapiero}). A crucial point in the non-periodic case will be to have a similar result, with the same
scaling in $\varepsilon$. This is why we use Assumption \textbf{(A2)}, which allows to prove Lemma~\ref{Poincmic2} below.

\begin{lemme}[Lemma 1 of \cite{tartar}]
\label{Poincmic}
There exists a constant $C_0> 0$ independent of $\varepsilon$ such that 
$$\forall u \in H^1_0(\Omega_{\varepsilon}), \ \| u \|_{L^2(\Omega_{\varepsilon})} \leq C_0 \varepsilon \| \nabla u \|_{L^2(\Omega_{\varepsilon})}.$$ 
\end{lemme}

This allows to prove Lemma~2 of \cite{tartar}:
\begin{lemme}[Lemma 2 of \cite{tartar}]
\label{lemimpt}
The solution $u_{\varepsilon}$ of Problem \eqref{poisson} satisfies the estimates
\begin{equation}
\| u_{\varepsilon} \|_{L^2(\Omega_{\varepsilon})} \leq C \varepsilon^2 \ \ \ \ \mathrm{and} \ \ \ \ \| u_{\varepsilon} \|_{H^1_0(\Omega_{\varepsilon})} \leq C \varepsilon,
\end{equation}
where $C$ is a constant independent of $\varepsilon$.
\end{lemme}

Using a two-scale expansion of the form (see \cite[Section 2]{lions-1980})
\begin{equation}
\label{2scale}
u_{\varepsilon}(x) = u_0^{\mathrm{per}} \left(x,\frac{x}{\varepsilon} \right) + \varepsilon u_1^{\mathrm{per}} \left(x, \frac{x}{\varepsilon} \right) + \varepsilon^2 u_2^{\mathrm{per}} \left(x, \frac{x}{\varepsilon} \right) + \varepsilon^3 u_3^{\mathrm{per}} \left(x, \frac{x}{\varepsilon} \right) + \cdots,
\end{equation}
where the functions $u_i^{\mathrm{per}}$ are defined on $\Omega \times (Q \setminus \overline{\mathcal{O}_0^{\mathrm{per}}})$, smooth and $Q-$periodic in the second variable, one proves that, at least formally,
\begin{equation}
  \label{eq:2-echelles}
  u_{\varepsilon}(x) = \varepsilon^2 w^{\mathrm{per}}\left( \frac{x}{\varepsilon} \right) f(x) + \cdots,
\end{equation}
where $w^{\mathrm{per}}$ is the periodic solution of 
\begin{equation}\label{eq:periodic_corrector}
\begin{cases}
\begin{aligned}
- \Delta w^{\mathrm{per}} & = 1, \\
w^{\mathrm{per}} & \in H^{1,\mathrm{per}}_0(Q \setminus \overline{\mathcal{O}^{\mathrm{per}}_0}).
\end{aligned}
\end{cases}
\end{equation}
We note that Problem \eqref{eq:periodic_corrector} is well-posed. Indeed, it suffices to apply Lax-Milgram's Lemma to the following variational form
$$\forall v \in H^{1,\mathrm{per}}_0(Q \setminus \overline{\mathcal{O}_0^{\mathrm{per}}}),  \quad \int_{Q \setminus \overline{\mathcal{O}_0^{\mathrm{per}}}} \nabla w^{\mathrm{per}} \cdot \nabla v = \int_{Q \setminus \overline{\mathcal{O}_0^{\mathrm{per}}}} v.$$

The following convergence result is proved in \cite[Theorem 3.1]{lions-1980} (take $m=2$ there).

\begin{theoreme}[Consequence of Theorem 3.1 of \cite{lions-1980}]
  Assume that $\mathcal{O}_0^{\mathrm{per}}$ is an open subset of $Q$ such that 
  $\mathcal{O}_0^{\mathrm{per}} \subset\subset Q$. Let $f \in \mathcal{D}(\Omega)$ and $u_{\varepsilon}$ be the solution to \eqref{poisson}. Then there exists a constant $C$ independent of $\varepsilon$ such that
\begin{equation} 
\label{CV}
\varepsilon^{-1}\left\| u_{\varepsilon} - \varepsilon^2 w^{\mathrm{per}}\left( \cdot / \varepsilon \right) f \right\|_{L^2(\Omega_{\varepsilon})} + \left| u_{\varepsilon} - \varepsilon^2 w^{\mathrm{per}}\left( \cdot / \varepsilon \right) f \right|_{H^1(\Omega_{\varepsilon})} \leq C \varepsilon^2,
\end{equation}
where $w^{\mathrm{per}} \in H^{1,\mathrm{per}}_0(Q \setminus \overline{\mathcal{O}^{\mathrm{per}}_0})$ is the unique function satisfying $- \Delta w^{\mathrm{per}} = 1$. 
\label{theoper}
\end{theoreme}

\begin{remarque}\label{rk1}
  If we assume in addition that $\mathcal{O}_0^{\mathrm{per}}$ is of class
$C^{1,\gamma}$ for some $0<\gamma<1$, then Theorem~\ref{theoper} still holds true under the weaker hypotheses $f \in H^2 \cap L^{\infty}(\Omega)$ and $f_{|\partial \Omega} =
0$ in the trace sense (see \cite[Appendix A.2]{lebris-legoll-lozinski-2014}). If we do not suppose that $f$ vanishes on $\partial \Omega$, $u_{\varepsilon} - \varepsilon^2 w( \cdot / \varepsilon) f$ does not vanish on $\partial \Omega$ either and we have the weaker estimate
$$\| u_{\varepsilon} - \varepsilon^2 w( \cdot / \varepsilon) f \|_{H^1(\Omega_\varepsilon)} \leq C \varepsilon^{3/2} \mathcal{N}(f),$$
where $\mathcal{N}(f) = \| f \|_{L^{\infty}} + \| \nabla f \|_{L^2} + \| \Delta f \|_{L^2}$. 
\end{remarque}

\subsection{The non-periodic case}
\label{sec:non-periodic-case}


We aim at extending the previous results to non-periodically perforated medium, in the special case of local
perturbations of the periodic structure.  More precisely, we define a reference periodic configuration by
\eqref{O_k^per}-\eqref{O_epsilon^per}-\eqref{domain}, and, for each $k \in \mathbb{Z}^d$, we denote by $\mathcal{O}_k$
the (non-periodic) perforation in the cell $k$. We assume that $\mathcal{O}_0^{\mathrm{per}}$ is of class
$C^{1,\gamma}$ for some $0<\gamma<1$. Our first assumption is that perforations should be sufficiently regular:\\

\noindent \textbf{(A1) For all $k \in \mathbb{Z}^d$, $\mathcal{O}_k$ is a $C^{1,\gamma}$ open set such that $\mathcal{O}_k \subset \subset Q_k$ and $Q_k \setminus \overline{\mathcal{O}_k}$ is connected.} \\

We now introduce geometric tools. For $\alpha > 0$, define the Minkowski-content of $\partial \mathcal{O}^{\mathrm{per}}_0$ (i.e a widened boundary of $\mathcal{O}_{0}^{\mathrm{per}}$) by the set
$$\mathcal{U}^{\mathrm{per}}_0(\alpha) := \{x \in \mathbb{R}^d \ \mathrm{s.t.} \ \mathrm{dist}(x,\partial \mathcal{O}^{\mathrm{per}}_0) < \alpha \}.$$
Similarly, if $k \in \mathbb{Z}^d$ and $\alpha > 0$, denote the set
$$\mathcal{U}_k^{\mathrm{per}}(\alpha) := \{x \in \mathbb{R}^d \ \mathrm{s.t.} \ \mathrm{dist}(x,\partial \mathcal{O}^{\mathrm{per}}_k) < \alpha \} = \mathcal{U}^{\mathrm{per}}_0(\alpha) + k.$$
Now (see Figure \ref{def} left), we define the reduction and the enlargement of $\mathcal{O}_k^{\mathrm{per}}$ by
$$\mathcal{O}_k^{\mathrm{per},-}(\alpha) := \mathcal{O}_k^{\mathrm{per}} \setminus \overline{\mathcal{U}_k^{\mathrm{per}}(\alpha)} \ \ \ \mathrm{and} \ \ \
\mathcal{O}_k^{\mathrm{per},+}(\alpha) := \mathcal{O}_k^{\mathrm{per}} \cup \mathcal{U}_k^{\mathrm{per}}(\alpha).$$
One has
$\displaystyle \mathcal{O}_k^{\mathrm{per},-}(\alpha) \subset \mathcal{O}_k^{\mathrm{per},+}(\alpha) \ \ \ \mathrm{and} \ \ \ \mathcal{U}_k^{\mathrm{per}}(\alpha) = \mathcal{O}_k^{\mathrm{per},+}(\alpha) \setminus \overline{\mathcal{O}_k^{\mathrm{per},-}(\alpha)}.$
We clearly have
\begin{equation}
\label{21}
\mathcal{O}_k^{\mathrm{per},+}(\alpha) = \{x \in \mathbb{R}^d \ \mathrm{s.t.} \ \mathrm{dist}(x,\mathcal{O}_k^{\mathrm{per}}) < \alpha \},
\end{equation}
and
\begin{equation}
\label{Omoins}
\mathcal{O}_k^{\mathrm{per},-}(\alpha) = \{x \in \mathcal{O}_k^{\mathrm{per}} \ \mathrm{s.t.} \ \mathrm{dist}(x, \partial \mathcal{O}_k^{\mathrm{per}}) > \alpha \}.
\end{equation}
Assumption \textbf{(A2)} reads :
\vspace{0.15cm}

\noindent \textbf{(A2) There exists a sequence $(\alpha_k)_{k \in \mathbb{Z}^d}$ such that $\alpha_k\geq 0,$  $(\alpha_k)_{k \in \mathbb{Z}^d} \in \ell^1(\mathbb{Z}^d)$ and
\begin{equation}
\label{HYP}
\forall k \in \mathbb{Z}^d, \  \mathcal{O}_k^{\mathrm{per},-}(\alpha_k) \subset \mathcal{O}_k \subset \mathcal{O}_k^{\mathrm{per},+}(\alpha_k)
\end{equation}
i.e $\mathcal{O}_k$ is between the enlargement and the reduction of $\mathcal{O}_k^{\mathrm{per}}$.}
\vspace{0.3cm}

\begin{remarque}
Assumption \textbf{(A2)} is a way to impose that the defect is localized. In \cite{BLLMilan,BLLcpde,BLLfutur1},
such an assumption is written as $a = a^{\rm per} + \widetilde a$, with $\widetilde a\in L^q(\RR^d)$, and $a^{\rm
  per}$ is periodic, where $a$ is the coefficient of the considered elliptic equation. Here, writing a similar condition, we
impose that the characteristic function of the perforations is a perturbation (i.e, a function in $L^q(\RR^d)$) of
the periodic case. For a characteristic function, being in $L^q(\RR^d)$ is equivalent to being in $L^1(\RR^d)$, hence the
condition \textbf{(A2)}.
\end{remarque}


 Note that if $\alpha_k$ is sufficiently large, $ \mathcal{O}_k^{\mathrm{per},-}(\alpha_k) = \emptyset$ and $Q_k
 \subset \mathcal{O}_k^{\mathrm{per},+}(\alpha_k)$. Thus, there is potentially no restriction on (a finite number
 of) $\mathcal{O}_k$. Figure \ref{def} (right) explains Assumption \textbf{(A2)}. \\
 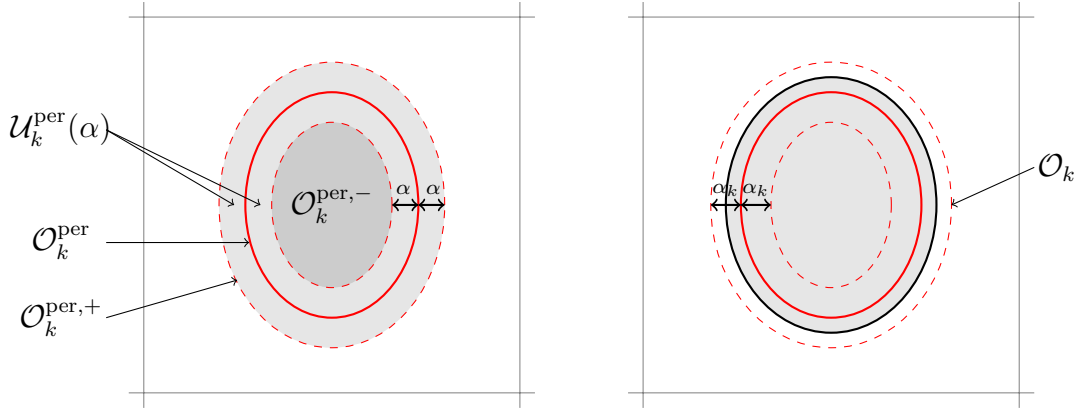
\begin{figure}[h!]
\centering
\begin{subfigure}{.5\textwidth}
  \centering
\begin{tikzpicture}[scale=1.]\footnotesize
 \pgfmathsetmacro{\xone}{-0.2}
 \pgfmathsetmacro{\xtwo}{5.2}
 \pgfmathsetmacro{\yone}{-0.2}
 \pgfmathsetmacro{\ytwo}{5.2}
\begin{scope}<+->;
  \draw[step=5cm,gray,very thin] (\xone,\yone) grid (\xtwo,\ytwo);
\end{scope}

\fill[gray!20] (2.5,2.5) ellipse (1.5cm and 1.9cm);
\fill[gray!40] (2.5,2.5) ellipse (0.8cm and 1.1cm);
\draw[red,thick] (2.5,2.5) ellipse (1.15cm and 1.5cm);
\draw[red,dashed] (2.5,2.5) ellipse (1.5cm and 1.9cm);
\draw[red,dashed] (2.5,2.5) ellipse (0.8cm and 1.1cm);
\draw[<->, thick] (3.65,2.5) -- (3.3,2.5);
\draw[<->, thick] (4,2.5) -- (3.65,2.5);

\draw (2.5,2.5) node[]{\large{$\mathcal{O}_k^{\mathrm{per},-}$}};
\draw (3.85,2.7) node[]{$\alpha$};
\draw (3.45,2.7) node[]{$\alpha$};
\draw[<-] (1.55,2.5) -- (-0.5,3.5); 
\draw[<-] (1.4,2) -- (-0.5,2);
\draw (-1.1,2) node[]{\large{$\mathcal{O}_k^{\mathrm{per}}$}};
\draw[<-] (1.24,1.5) -- (-0.5,1);
\draw (-1.1,1) node[]{\large{$\mathcal{O}_k^{\mathrm{per},+}$}};
\draw[<-] (1.2,2.5) -- (-0.5,3.5);
\draw (-1.1,3.5) node[]{\large{$\mathcal{U}_k^{\mathrm{per}}(\alpha)$}};

\end{tikzpicture}
\label{defa}
\end{subfigure}%
\begin{subfigure}{.5\textwidth}
  \centering
\begin{tikzpicture}[scale=1.]\footnotesize
 \pgfmathsetmacro{\xone}{-0.2}
 \pgfmathsetmacro{\xtwo}{5.2}
 \pgfmathsetmacro{\yone}{-0.2}
 \pgfmathsetmacro{\ytwo}{5.2}
\begin{scope}<+->;
  \draw[step=5cm,gray,very thin] (\xone,\yone) grid (\xtwo,\ytwo);
\end{scope}

\fill[gray!20] (2.5,2.5) ellipse (1.4cm and 1.7cm);
\draw[red,thick] (2.5,2.5) ellipse (1.2cm and 1.5cm);
\draw[black,thick] (2.5,2.5) ellipse (1.4cm and 1.7cm);
\draw[red,dashed] (2.5,2.5) ellipse (1.6cm and 1.9cm);
\draw[red,dashed] (2.5,2.5) ellipse (.8cm and 1.1cm);
\draw (1.5,2.7) node[]{$\alpha_k$};
\draw (1.1,2.7) node[]{$\alpha_k$};
\draw[<->, thick] (0.9,2.5) -- (1.3,2.5);
\draw[<->, thick] (1.3,2.5) -- (1.7,2.5);
\draw[<-] (4.1,2.5) -- (5.2,3);
\draw (5.5,3) node[]{\large{$\mathcal{O}_k$}};

\end{tikzpicture}
\label{defb}
\end{subfigure}%
\caption{On the left, illustration of $\mathcal{O}_k^{\mathrm{per}}$ (red), its widened boundary $\mathcal{U}_k^{\mathrm{per}}$, its enlargement $\mathcal{O}_k^{\mathrm{per},+}$ and its reduction $\mathcal{O}_k^{\mathrm{per},-}$ (grey). On the right, $\mathcal{O}_k$.}
\label{def}
\end{figure}

We define
\begin{equation}\label{eq:def_O}
\mathcal{O} := \bigcup_{k \in \mathbb{Z}^d} \mathcal{O}_k.
\end{equation}
We split the domain $\mathbb{R}^d \setminus \overline{\mathcal{O}}$ into two subdomains:
$$\mathbb{R}^d \setminus (\overline{\mathcal{O} \cup \mathcal{O}^{\mathrm{per}}}) \ \ \ \mathrm{and} \ \ \ \mathcal{O}^{\mathrm{per}} \setminus \overline{\mathcal{O}}.$$ Note that these domains are not necessarily connected. 

We split the boundary of the domain $\mathcal{O}^{\mathrm{per}} \setminus \mathcal{O}$ into two parts (the one surrounding $\mathcal{O}^{\mathrm{per}}$ and the one surrounding $\mathcal{O}$). For $k \in \mathbb{Z}^d$, we define
\begin{equation}
\label{gamma1}
\Gamma_1^{k} = \partial \mathcal{O}_k^{\mathrm{per}} \setminus \overline{\mathcal{O}_k} \ \ \ \mathrm{and} \ \ \ \Gamma_2^k = \partial \mathcal{O}_k \cap \mathcal{O}_k^{\mathrm{per}}\ \ \ \mathrm{s.t} \ \ \ \partial(\mathcal{O}^{\mathrm{per}}_k \setminus \overline{\mathcal{O}_k} ) = \Gamma_1^k \cup \Gamma_2^k.
\end{equation}
We denote by $\Gamma_1$ (resp. $\Gamma_2$) the union of the $\Gamma_1^k$ (resp. $\Gamma_2^k$), $k \in \mathbb{Z}^d$:
\begin{equation}
  \label{gamma12}
  \Gamma_1 = \bigcup_{k\in\ZZ^d} \Gamma_1^k, \quad \Gamma_2 = \bigcup_{k\in\ZZ^d} \Gamma_2^k.
\end{equation}

We also split the boundary of $\mathbb{R}^d \setminus (\overline{\mathcal{O} \cup \mathcal{O}^{\mathrm{per}}})$ into two parts. Write 
$\displaystyle
\partial(\mathbb{R}^d \setminus (\overline{\mathcal{O} \cup \mathcal{O}^{\mathrm{per}}})) = \partial(\mathcal{O} \cup \mathcal{O}^{\mathrm{per}}),
$
and define for $k \in \mathbb{Z}^d$ 
\begin{equation}
\label{gamma3}
\Gamma_3^k = \partial \mathcal{O}_k \setminus \mathcal{O}_k^{\mathrm{per}} \ \ \mathrm{s.t} \ \ \partial(\mathcal{O}_k \cup \mathcal{O}_k^{\mathrm{per}}) = \Gamma_1^k \cup \Gamma_3^k.
\end{equation}
Note that
$$\partial \mathcal{O}_k = \Gamma_2^k \cup \Gamma_3^k.$$
$\Gamma_3$ denotes the union of the $\Gamma_3^k$. Note that $\Gamma_3$ is in fact the complement of $\Gamma_2$ in
$\partial \mathcal{O}$. Figure \ref{figproof} explains the above definitions. \\

\begin{figure}[h!]
\centering
\begin{subfigure}{.5\textwidth}
  \centering
\begin{tikzpicture}[scale=.7]\footnotesize
 \pgfmathsetmacro{\xone}{-1}
 \pgfmathsetmacro{\xtwo}{6}
 \pgfmathsetmacro{\yone}{-1}
 \pgfmathsetmacro{\ytwo}{6}
\begin{scope}<+->;
  \draw[step=5cm,gray,very thin] (\xone,\yone) grid (\xtwo,\ytwo);
\end{scope}

\draw[opacity=0.9,black,ultra thick] (3,3) ellipse (1.3cm and 1cm);
\fill[white] (2,2) ellipse (0.9cm and 1.2cm);
\fill[gray!20] (2,2) ellipse (0.9cm and 1.2cm);
\draw[blue,semithick] (2,2) ellipse (0.9cm and 1.2cm);
\draw (2,1.5) node[]{$\mathcal{O}_k^{\mathrm{per}}\setminus \mathcal{O}_k$};
\fill[gray!60] (3,3) ellipse (1.3cm and 1cm);
\draw[opacity=0.4,red,ultra thick] (3,3) ellipse (1.3cm and 1cm);
\draw (3.5,3) node[]{\large{$\mathcal{O}_k$}};
\draw[blue] (0.8,2) node[]{\large{$\Gamma_1^k$}};
\draw[black] (3,4.3) node[]{\large{$\Gamma_3^k$}};
\draw[red] (2.2,2.7) node[]{\large{$\Gamma_2^k$}};

\end{tikzpicture}
\end{subfigure}%
\begin{subfigure}{.5\textwidth}
  \centering
\begin{tikzpicture}[scale=.7]\footnotesize
 \pgfmathsetmacro{\xone}{-1}
 \pgfmathsetmacro{\xtwo}{6}
 \pgfmathsetmacro{\yone}{-1}
 \pgfmathsetmacro{\ytwo}{6}
\begin{scope}<+->;
  \draw[step=5cm,gray,very thin] (\xone,\yone) grid (\xtwo,\ytwo);
\end{scope}

\fill[gray!20] (2,2) ellipse (0.9cm and 1.2cm);
\draw[blue,semithick] (2,2) ellipse (0.9cm and 1.2cm);
\draw (2,1.5) node[]{$\mathcal{O}_k^{\mathrm{per}}\setminus \mathcal{O}_k$};
\fill[gray!60] (3.7,3.7) ellipse (1cm and 1cm);
\draw[opacity=0.9,black,very thick] (3.7,3.7) ellipse (1cm and 1cm);
\draw[opacity=0.4,red,very thick] (3.7,3.7) ellipse (1cm and 1cm);
\draw (3.7,3.7) node[]{\large{$\mathcal{O}_k$}};
\draw[blue] (0.8,2) node[]{\large{$\Gamma_1^k$}};
\draw[black] (3.7,2.4) node[]{\large{$\Gamma_3^k$}};

\end{tikzpicture}

\end{subfigure}
\caption{Pictures of perforated cells divided into two subdomains (white and light grey) with boundary $\Gamma_i,
  i=1,2,3$. Left: $\mathcal{O}_k\cap \mathcal{O}_k^{\mathrm{per}} \neq \emptyset$. Right: $\mathcal{O}_k\cap \mathcal{O}_k^{\mathrm{per}} = \emptyset$}
\label{figproof}
\end{figure}
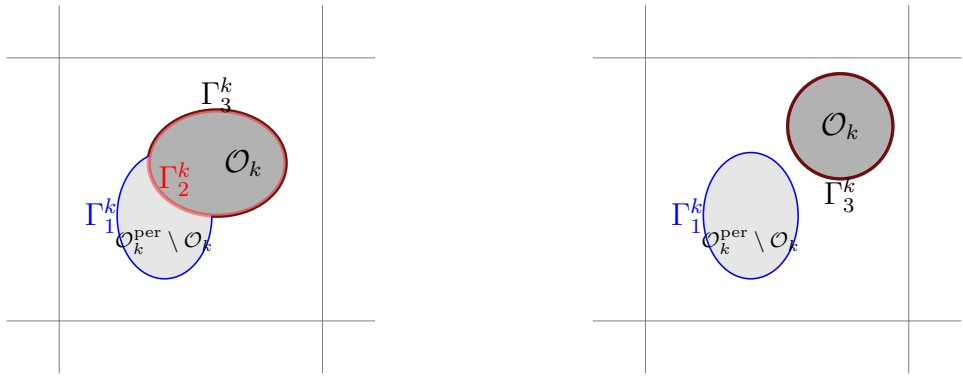

We deduce from Assumption \textbf{(A2)} Lemma~\ref{H1}, \ref{H3} and \ref{H2}, which are stated and proved in
Appendix~\ref{sec:appendix}.


\section{Results}\label{sec:results}

In order to state our main result, we first need to prove that a corrector exists:
\begin{theoreme} [Existence and uniqueness of the corrector] Let $(\mathcal{O}_k)_{k \in
    \mathbb{Z}^d}$ be a sequence of open sets satisfying Assumptions \textbf{(A1)-(A2)}. Let $\mathcal O$ be
  defined by \eqref{eq:def_O}, and
$$g = \11_{\RR^d\setminus\overline{\mathcal{O}^{\mathrm{per}}}} + \widetilde{g}, \quad\text{with}\quad \widetilde{g} \in  L^2(\mathbb{R}^d).$$ There exists a unique $\widetilde{w} \in H^1(\mathbb{R}^d \setminus \overline{\mathcal{O}})$ such that $w := w^{\mathrm{per}} + \widetilde{w} \in H^{1,\mathrm{per}}(Q) + H^1(\mathbb{R}^d \setminus \overline{\mathcal{O}})$ is solution in the sense of distributions of
\begin{equation}
\begin{cases}
\begin{aligned}
- \Delta w & = g \ \mathrm{in} \ \mathbb{R}^d \setminus \overline{\mathcal{O}} \\
w_{| \partial \mathcal{O}}&  = 0,
\end{aligned}
\end{cases}
\label{wtilde}
\end{equation}
where $w^{\mathrm{per}} \in H^{1,\mathrm{per}}_0(Q \setminus \overline{\mathcal{O}_0^{\mathrm{per}}})$ is the unique solution
of the periodic corrector problem \eqref{eq:periodic_corrector} extended by zero to $\mathbb{R}^d$.
\label{theocor}
\end{theoreme} 

Using Theorem~\ref{theocor} and a two-scale expansion, as it is done in the periodic case, we have the following
result, which is the generalization of Theorem~\ref{theoper} to the present setting
\begin{theoreme}[Convergence theorem in $H^1-$ norm]

Let $(\mathcal{O}_k)_{k \in \mathbb{Z}^d}$ be a sequence of open sets satisfying Assumptions \textbf{(A1)-(A2)},
and assume that $\mathcal O$ is defined by \eqref{eq:def_O}.

Let $\Omega \subset \mathbb{R}^d$ be a bounded domain and define for $\varepsilon > 0$ the perforated set $\Omega_{\varepsilon} := \Omega \setminus \varepsilon\overline{\mathcal{O}}$.

Let $f \in \mathcal{D}(\Omega)$ and $u_{\varepsilon}$ be the solution of Problem \eqref{poisson}. Then there exists a constant $C > 0$ independent of $\varepsilon$ such that
\begin{equation} 
\label{CV2}
\varepsilon^{-1} \left\| u_{\varepsilon} - \varepsilon^2 w \left( \cdot / \varepsilon \right) f \right\|_{L^2(\Omega_{\varepsilon})} + \left| u_{\varepsilon} - \varepsilon^2 w \left( \cdot / \varepsilon \right) f \right|_{H^1(\Omega_{\varepsilon})} \leq C \varepsilon^2, 
\end{equation}
where $w = w^{\mathrm{per}} + \widetilde{w} \in H^{1,\mathrm{per}}(Q) + H^1(\mathbb{R}^d \setminus
\overline{\mathcal{O}})$ is the unique solution of the corrector Problem \eqref{wtilde} with $g=1$.
\label{theo}
\end{theoreme}

We note that the constant $C$ appearing in Theorem~\ref{theo} is independent of $\varepsilon$ but depends on $f$, on the non-periodic corrector constructed in Theorem~\ref{theocor} and on the Poincar\'e-Friedrichs constant of $\Omega_{\varepsilon}$ (denoted $C$ in Lemma~\ref{Poincmic2} below).

Theorem \ref{theo} provides an error estimate of $u_{\varepsilon} - \varepsilon^2 w(\cdot/\varepsilon) f$ in
$H^1(\Omega_{\varepsilon})-$norm. However, for this choice of norm, the use of a non-periodic corrector appears
to be irrelevant, which means that we could also have used the periodic corrector $w^{\mathrm{per}}$ in \eqref{CV2}
without changing the rate of convergence. Indeed, we have
\begin{equation}
  \label{eq:1}
\left\| \varepsilon^2  \widetilde{w} \left(\frac{\cdot}{\varepsilon} \right) f \right\|_{H^1(\Omega_{\varepsilon})}
= O\left(\varepsilon^2\right).
\end{equation}
In order to prove \eqref{eq:1}, we only deal with the leading order term of the above quantity, that is, the $L^2-$norm of the gradient. One has 
\begin{multline*}
\int_{\Omega_\varepsilon}\left|\nabla \left[ \varepsilon^2 \widetilde{w} \left( \frac{\cdot}{\varepsilon} \right) f \right](x) \right|^2 
\leq 
2\varepsilon^2 \int_{\Omega_\varepsilon}\left| \nabla \widetilde{w} \left( \frac{x}{\varepsilon} \right) \right|^2
|f(x)|^2 + 2\varepsilon^4 \int_{\Omega_\varepsilon}\widetilde{w} \left( \frac{x}{\varepsilon} \right)^2 |\nabla
f(x) |^2\\ \leq C \varepsilon^2 \int_{\Omega_{\varepsilon}} \left|\nabla \widetilde{w} \left( \frac{x}{\varepsilon} \right) \right|^2 \mathrm{d}x + C \varepsilon^4 \int_{\Omega_{\varepsilon}} \left| \widetilde{w} \left( \frac{x}{\varepsilon} \right) \right|^2 \mathrm{d}x.  
\end{multline*}
Thus, after the change of variable $y = x / \varepsilon$,
$$
\int_{\Omega_{\varepsilon}} \left|\nabla \left[ \varepsilon^2 \widetilde{w} \left( \frac{\cdot}{\varepsilon} \right) f \right](x) \right|^2 \mathrm{d}x
\leq 
C \varepsilon^{d+2} \int_{\mathbb{R}^d \setminus \mathcal{O}} |\nabla \widetilde{w}(y)|^2\mathrm{d}y + C \varepsilon^{d+4} \int_{\mathbb{R}^d \setminus \mathcal{O}} | \widetilde{w}(y)|^2 \mathrm{d}y.
$$
We thus have \eqref{eq:1}, which implies (since $d\geq 2$)
$$\left\| u_{\varepsilon} - \varepsilon^2 w^{\mathrm{per}} \left( \frac{\cdot}{\varepsilon} \right) f \right\|_{H^1(\Omega_{\varepsilon})} \leq \left\| u_{\varepsilon} - \varepsilon^2 w \left( \frac{\cdot}{\varepsilon} \right) f \right\|_{H^1(\Omega_{\varepsilon})}  + \left\| \varepsilon^2  \widetilde{w} \left(\frac{\cdot}{\varepsilon} \right) f \right\|_{H^1(\Omega_{\varepsilon})} = O\left(\varepsilon^2 \right).$$
Thus, using $w^{\mathrm{per}}$ instead of $w$ in convergence Theorem \ref{theo} does not change the order $O(\varepsilon^2)$ of the error. \\

The following Theorem states that the use of $w$ instead of $w^{\mathrm{per}}$ improves the rate of convergence in $L^{\infty}-$norm for a non-periodic domain. 

\begin{theoreme}[Convergence Theorem in $L^{\infty}-$norm]
Let $(\mathcal{O}_k)_{k \in \mathbb{Z}^d}$ be a sequence of open sets satisfying Assumptions \textbf{(A1)-(A2)},
and assume that $\mathcal{O}$ is defined by \eqref{eq:def_O}. Assume that the $C^{1,\gamma}$ norms of the charts that flatten $\partial \mathcal{O}_k$ are uniformly bounded in $k$.

Let $\Omega \subset \mathbb{R}^d$ be a bounded domain and define for $\varepsilon > 0$ the perforated set $\Omega_{\varepsilon} := \Omega \setminus \varepsilon\overline{\mathcal{O}}$.

Let $f \in \mathcal{D}(\Omega)$ and $u_{\varepsilon}$ be the solution of \eqref{poisson}. Then there exists a constant $C > 0$ independent of $\varepsilon$ such that
$$
\left\| u_{\varepsilon} - \varepsilon^2 w \left( \cdot / \varepsilon \right) f \right\|_{L^{\infty}(\Omega_{\varepsilon})} \leq C \varepsilon^3, 
$$
where $w = w^{\mathrm{per}} + \widetilde{w} \in H^1_{\mathrm{per}}(Q) + H^1(\mathbb{R}^d \setminus
\overline{\mathcal{O}})$ is the unique solution of \eqref{wtilde} with $g=1$.
\label{theoinfini}
\end{theoreme}

Note that $\| \varepsilon^2 \widetilde{w}(\cdot/\varepsilon) f \|_{L^{\infty}(\Omega_{\varepsilon})}$ is generally of order $\varepsilon^2$ exactly. \\ Fix $K \subset \Omega$. One has 
$$\left\| \left[ \varepsilon^2 \widetilde{w} \left( \frac{\cdot}{\varepsilon} \right) f \right](\varepsilon \cdot) \right\|_{L^{\infty}(K)} \underset{\varepsilon \rightarrow 0}{\sim} \varepsilon^2 f(0) \| \widetilde{w} \|_{L^{\infty}(K)}.$$
Besides, Theorem \ref{theoinfini} implies
$$\left\| \left[ u_{\varepsilon} - \varepsilon^2 w\left( \frac{\cdot}{\varepsilon} \right) f \right](\varepsilon \cdot) \right\|_{L^{\infty}(K)} \leq C \varepsilon^3.$$
Thus, 
$$\left\| \left[ u_{\varepsilon} - \varepsilon^2 w^{\mathrm{per}} \left( \frac{\cdot}{ \varepsilon} \right) f \right](\varepsilon \cdot) \right\|_{L^{\infty}(K)} \sim C \varepsilon^2.$$
We have the same results for $L^{\infty}(K)-$norm replaced by $L^2(K)-$norm. This proves that convergence of $u_{\varepsilon}/\varepsilon^2 - w(\cdot / \varepsilon) f$ holds at the microscale in $L^2-$norm when we use $w$. This is not the case when we use the periodic corrector $w^{\mathrm{per}}$. 

\begin{remarque}\label{rk3} This Remark is analogous to Remark~\ref{rk1} in the present non-periodic setting.
The condition $f \in \mathcal{D}(\Omega)$ can be weakened in Theorem~\ref{theo} provided that we use Lemma~\ref{lem14} proved below. Under H\"{o}lder regularity conditions on the perforations, one has thanks to Lemma \ref{lem14} that $w \in L^{\infty}(\mathbb{R}^d \setminus \overline{\mathcal{O}})$ and $\nabla w \in L^{\infty}(\mathbb{R}^d \setminus \overline{\mathcal{O}})$. Thus, if we suppose that $f \in H^2(\Omega)$ and $f_{|\partial \Omega} = 0$ in the trace sense, we obtain (see \eqref{227}),
$$\| g_{\varepsilon} \|_{L^2(\Omega_{\varepsilon})} \leq 2 \| \nabla w\|_{L^{\infty}} \| \nabla f \|_{L^2}  + \| w \|_{L^{\infty}} \| \Delta f\|_{L^2} \leq C \| f \|_{H^2(\Omega)}
$$ 
for $\varepsilon < 1$. We deduce by integration by parts that $\| u_{\varepsilon} - \varepsilon^2 w(\cdot/\varepsilon) f \|_{H^1(\Omega_{\varepsilon})} \leq C \varepsilon^2$.

If $f$ does not vanish on $\partial \Omega$, we can prove that there exists a constant $C$ independent of $\varepsilon$ such that
 $$\| u_{\varepsilon} - \varepsilon^2 w(\cdot/\varepsilon) f \|_{H^1(\Omega_{\varepsilon})} \leq C \varepsilon^{3/2} \mathcal{N}(f).$$
The proof is analogous to \cite[Appendix A.2]{lebris-legoll-lozinski-2014} provided we use Lemma \ref{lem14} below. This requires $f \in H^2 \cap L^{\infty}(\Omega)$.
\end{remarque}


\section{Poincar\'e-Friedrichs inequalities}\label{sec:poincare}

The main ingredient of the proof of Theorem~\ref{theocor} is the following Poincar\'e-type inequality.

\begin{theoreme}
\label{unifpoinc} 
Let $Q$ be the unit cube of $\mathbb{R}^d$ and let $U$ be an open subset of $Q$ containing a box $\mathcal{R} = \prod_{i=1}^d [a_i,b_i]$. Then
\begin{equation}
\forall v \in H^1(Q \setminus\overline{U}) \ \ \mathrm{s.t} \ \ v_{|\partial U} = 0, \ \int_{Q \setminus \overline{U}} |v|^2 \leq \frac{d}{|\mathcal{R}|} \int_{Q \setminus \overline{U}} |\nabla v|^2.
\label{Poincareequation}
\end{equation}
Similarly,
$$
\forall v \in H^1(Q) \ \mathrm{s.t} \ \ v_{|U} = 0, \ \int_{Q} |v|^2 \leq \frac{d}{|\mathcal{R}|} \int_{Q} |\nabla v|^2.
$$
\end{theoreme}

An important point in \eqref{Poincareequation}, is that the constant is explicit and depends only on $\mathcal
R$. This crucial point will allow us, with the help of Assumption \textbf{(A2)}, to prove Lemma~\ref{Poincmic2}
below, in which the fundamental point is that the constant does not depend on $\varepsilon$. We thus have an
explicit scaling with respect to $\varepsilon$, similarly to the periodic case. This allows us to adapt the proofs
of \cite{lions-1980}.

\begin{proof}
By density, it is enough to show the result for $v \in C^1(Q)$ satisfying $v=0$ on $U$. Fix $x \in Q$ and write
$$
v(x) - v(\hat{x}) = \int_0^1 \nabla v ( (1-t)\hat{x} + t x)\cdot ( x - \hat{x} ) \mathrm{d}t,
$$
where $\hat{x} = (a_i + x_i(b_i - a_i) )_{1 \leq i \leq d} \in \mathcal{R}$. Note that $v(\hat{x}) = 0$ and $|x - \hat{x}|^2 \leq d$. Thus by the Cauchy-Schwarz inequality
$\displaystyle
|v(x)|^2 \leq d \int_0^1 \left| \nabla v ( (1-t)\hat{x} + t x) \right|^2 \mathrm{d}t.
$
Integrating with respect to $x \in Q$ and exchanging the two integrals yields
$$
\int_Q |v(x)|^2 \mathrm{d}x \leq d \int_0^1 \left( \int_Q \left| \nabla v ( (1-t)\hat{x} + t x) \right|^2 \mathrm{d}x \right) \mathrm{d}t
$$
Fix $t \in [0,1]$ and define the diffeomorphism $\phi_t : Q \ni x \mapsto (1-t)\hat{x} + tx$. Note that $\phi_t(Q) \subset Q$ and that
$\displaystyle
|\det J(\phi_t)| = \prod_{i=1}^d \left[ (1-t)(b_i - a_i) + t \right] \geq \prod_{i=1}^d (b_i - a_i).
$
Thus by a change of variables,
$$
\int_Q \left| \nabla v ( (1-t)\hat{x} + t x) \right|^2 \mathrm{d}x \leq \frac{1}{\prod_{i=1}^d (b_i - a_i)} \int_Q |\nabla v|^2.
$$
Integrating with respect to $t$ concludes the proof.
\end{proof}

Theorem~\ref{unifpoinc} and Assumption \textbf{(A2)} allow to prove the following, which is a generalization to the present setting of Lemma~\ref{Poincmic}.

\begin{lemme}[Poincar\'e-type inequality in $H^1_0(\Omega_{\varepsilon}$)]
\label{Poincmic2}
Let $(\mathcal{O}_k)_{k \in \mathbb{Z}^d}$ be a sequence of open sets such that $\mathcal{O}_k \subset\subset Q_k$. Suppose that the sequence $(\mathcal{O}_k)_{k \in \mathbb{Z}^d}$ satisfies Assumption \textbf{(A2)}. Let $\Omega$ be an open subset of $\mathbb{R}^d$. Define for $\varepsilon > 0$,
$$\Omega_{\varepsilon} = \Omega \setminus \varepsilon \overline{\mathcal{O}} = \Omega \cap \left( \bigcup_{k \in \mathbb{Z}^d} \varepsilon ( Q_k \setminus \overline{\mathcal{O}_k}) \right).$$
There exists a constant $C > 0$ independent of $\varepsilon$ such that
$$
\forall u \in H^1_0(\Omega_{\varepsilon}), \ \int_{\Omega_{\varepsilon}} u^2 \leq C \varepsilon^2 \int_{\Omega_{\varepsilon}} |\nabla u|^2.
$$
\end{lemme}

\begin{proof}
We first recall (see Lemma \ref{H2} in the appendix) that $\mathcal{K} := \{k \in \mathbb{Z}^d, \quad \mathcal{O}_k \cap \mathcal{O}_k^{\mathrm{per}} = \emptyset \}$ is finite.
We show that there exists $\widetilde{\rho} > 0$ such that for all $k \in \mathbb{Z}^d$, there exists a box $\mathcal{R}_k \subset \mathcal{O}_k$ satisfying $|\mathcal{R}_k| \geq \widetilde{\rho}$. Fix $k \in \mathbb{Z}^d$, there are two cases :
\begin{itemize}
\item \textit{Case 1}: $k \in \mathcal{K}$ (see Lemma \ref{H1}). The open set $\mathcal{O}_k$ contains a ball and thus a box $\mathcal{R}_k$.
\item \textit{Case 2}: $k \notin \mathcal{K}$. By Lemma \ref{H3}, there exists a ball $B_k \subset \mathcal{O}_k$ such that $|B_k| \geq \rho$ with $\rho$ independent of $k$. Thus, there exists a box $\mathcal{R}_k \subset \mathcal{O}_k$ such that $|\mathcal{R}_k| \geq C(d) \rho$ where $C(d)$ is a constant depending only on $d$. 
\end{itemize}
We define $\widetilde{\rho} := \min \left( C(d) \rho, \min\limits_{k \in \mathcal{K}} |\mathcal{R}_k| \right) > 0$ to conclude.

 We next use Theorem \ref{unifpoinc}. We get that
$$
\forall k \in \mathbb{Z}^d, \ \forall w \in 
H^1_0\left(\RR^d\setminus \overline{\mathcal O}\right),
\ \ \int_{Q_k \setminus \overline{\mathcal{O}_k}} w^2 \leq \frac{d}{\widetilde{\rho}} \int_{Q_k \setminus \overline{\mathcal{O}_k}} |\nabla w|^2.
$$
Summing over $k \in \mathbb{Z}^d$ each inequality and defining $C := d/\widetilde{\rho}$ yields
\begin{equation}
\label{Po}
\forall w \in H^1_0(\mathbb{R}^d \setminus \overline{\mathcal{O}}), \ \int_{\mathbb{R}^d \setminus \overline{\mathcal{O}}} w^2 \leq C \int_{\mathbb{R}^d \setminus \overline{\mathcal{O}}} |\nabla w|^2.
\end{equation}

Now, fix $u \in H^1_0(\Omega_{\varepsilon})$. We extend $u$ by zero to $\mathbb{R}^d \setminus \varepsilon \overline{\mathcal{O}}$ and define $v := u(\varepsilon \cdot)$. It is clear that $v \in H^1_0(\mathbb{R}^d \setminus \overline{\mathcal{O}})$ and that
\begin{equation}
\label{213}
\forall y \in \mathbb{R}^d \setminus \overline{\mathcal{O}}, \ \nabla v\left(y \right) =\varepsilon \nabla u(\varepsilon y).
\end{equation}
Applying \eqref{Po} with $w = v \in H^1_0(\mathbb{R}^d \setminus \overline{\mathcal{O}})$ and using \eqref{213} yields
$$\int_{\frac{1}{\varepsilon} \Omega_{\varepsilon}} u^2(\varepsilon y) \mathrm{d}y = \int_{\mathbb{R}^d \setminus \overline{\mathcal{O}}} u^2\left(\varepsilon y \right) \mathrm{d}y \leq C\varepsilon^2 \int_{\mathbb{R}^d \setminus \overline{\mathcal{O}}} |\nabla u|^2\left(\varepsilon y \right) \mathrm{d}y = C \varepsilon^2 \int_{\frac{1}{\varepsilon} \Omega_{\varepsilon}} |\nabla u|^2(\varepsilon y) \mathrm{d}y.$$
Making the change of variables $x=\varepsilon y$ in each integral finally concludes the proof.

\end{proof}

\section{Proofs}\label{sec:proofs}

\subsection{Two-scale expansion}

The aim of this section is to find an asymptotic equivalent of $u_{\varepsilon}$ as $\varepsilon$ goes to zero. We begin by the two scale expansion of $u_{\varepsilon}$. Write \\
$$
u_{\varepsilon}(x) = u_0 \left(x,\frac{x}{\varepsilon} \right) + \varepsilon u_1 \left(x, \frac{x}{\varepsilon} \right) + \varepsilon^2 u_2 \left(x, \frac{x}{\varepsilon} \right) + \varepsilon^3 u_3 \left(x, \frac{x}{\varepsilon} \right) + \cdots,
$$
where the functions $u_i$ are now defined on $\Omega \times (\mathbb{R}^d \setminus \overline{\mathcal{O}})$,
 and are of the form $u_i^{\mathrm{per}} + \widetilde{u_i}$. Suppose that $\widetilde{u_i}(x,\cdot) \in
H^1(\mathbb{R}^d \setminus \overline{\mathcal{O}})$ and use the $u_i^{\mathrm{per}}$'s defined in
Section~\ref{sec:periodic-case} and extended by zero to $\mathbb{R}^d$. Because of the Dirichlet Boundary
conditions on $u_{\varepsilon}$, we impose that $u_i(x,y) = 0$ for $y \in \partial \mathcal{O}$ and any $x \in
\Omega$. The calculations leading to \eqref{eq:2-echelles} (see \cite[Section 2]{lions-1980}) are still valid, so we have:
\begin{equation}
\label{eqperturbe}
\begin{cases}
- \Delta_y u_0 = 0 \\
- \Delta_y u_1 - 2(\nabla_x \cdot \nabla_y) u_0 = 0 \\
- \Delta_y u_2 - 2(\nabla_x \cdot \nabla_y) u_1 - \Delta_x u_0 = f \\
- \Delta_y u_3 - 2(\nabla_x \cdot \nabla_y) u_2 - \Delta_x u_1 = 0 \\
\cdots
\end{cases},
\end{equation}
where all these equations are posed on $\Omega \times (\mathbb{R}^d  \setminus \overline{\mathcal{O}})$. These equations imply that $u_0$ and $u_1$ are constantly equal to zero. Indeed, fix $x \in \Omega$. Since $u_0^{\mathrm{per}} \equiv 0$, we get that $\widetilde{u_0}(x,\cdot)$ satisfies the PDE 
$$
\begin{cases}
\begin{aligned}
- \Delta_y \widetilde{u_0} & = 0 \ \mathrm{in} \ \mathbb{R}^d \setminus \overline{\mathcal{O}}, \\
\widetilde{u_0}_{|\partial \mathcal{O}}  & = 0.
\end{aligned}
\end{cases}
$$
Multiplying by $\widetilde{u_0}(x,\cdot) \in H^1(\mathbb{R}^d \setminus \overline{\mathcal{O}})$ and integrating by parts yields $\widetilde{u_0}(x,\cdot) \equiv 0$. Thus $u_0 \equiv 0$. Similarly, $u_1 \equiv 0$. We are now left with the following equation on $u_2$ :

\begin{equation}
\label{u2pert}
\begin{cases}
\begin{aligned}
- \Delta_y u_2(x,y) & = f(x) \ \ \mathrm{in} \ \ \mathbb{R}^d \setminus \overline{\mathcal{O}} \\ 
u_2(x,y) & = 0, \ \ x \in \Omega, y \in \partial \mathcal{O}.
\end{aligned}
\end{cases}
\end{equation}

According to (\ref{u2pert}), $u_2(x,y) = f(x) w(y)$, where $w$ is a solution to the corrector equation
(\ref{wtilde}) with $g\equiv 1$. This is why we introduced the corrector equation. 
\subsection{Proof of the existence of a corrector}

The aim of this section is to prove Theorem \ref{theocor}. The difficulty of this theorem is that equation \eqref{wtilde} is posed on an unbounded domain. 

We search for $w$ in the form $w^{\mathrm{per}} + \widetilde{w}$, where we impose that $\widetilde{w} \in H^1(\mathbb{R}^d \setminus \overline{\mathcal{O}})$. We write the equation on $\widetilde{w}$ and prove by energy minimization that there is a solution. 

\subsubsection{Perturbed corrector}

The equation we want to solve for $\widetilde{w}$ is 
\begin{equation}
\label{wtilde2}
- \Delta \widetilde{w} = \11_{\RR^d\setminus \overline{\mathcal O}_{\mathrm{per}}} + \widetilde{g} + \Delta w^{\mathrm{per}},
\end{equation}
where $\widetilde{g} \in L^2(\mathbb{R}^d)$ and $w^{\mathrm{per}} \in H^{1,\mathrm{per}}(Q)$ is the solution to \eqref{eq:periodic_corrector} defined in Section~\ref{Introduction}. We recall that $w^{\mathrm{per}}$ is extended by zero in $\mathcal{O}^{\mathrm{per}}$. We impose that $\widetilde{w} = - w^{\mathrm{per}}$ on $\partial \mathcal{O}$.

It is worth noticing that $w^{\mathrm{per}} \notin H^2(Q)$, and thus the right-hand side of \eqref{wtilde2} cannot
be in $L^2(\mathbb{R}^d \setminus \overline{\mathcal{O}})$. Thus the linear form of the weak formulation of
\eqref{wtilde2} is not of the form $v \mapsto \int f v$. In fact, we will have to deal with boundary terms along
$\partial  \mathcal{O}^{\mathrm{per}}$. These terms express the fact that $\Delta w_{\mathrm{per}}$ is a Dirac
measure on $\partial \mathcal{O}^{\mathrm{per}}$ (or that $w^{\mathrm{per}}$ has normal derivative jumps along
$\partial \mathcal{O}^{\mathrm{per}}$). 

\paragraph{Notation.} We denote by $\left.\frac{\partial u}{\partial n}\right|_{\mathrm{ext}}$ (resp. $\left.\frac{\partial u}{\partial n}\right|_{\mathrm{int}}$ ) the exterior normal derivative of $u$ on the outside (resp. inside) of a piecewise smooth closed surface $\Gamma$ (when it is defined i.e $u$ is $H^2$ on each side of the boundary). 

\begin{definition}
We say that $\widetilde{w} \in H^1(\mathbb{R}^d \setminus \overline{\mathcal{O}})$ is a weak solution of \eqref{wtilde2} if
\begin{equation}
\label{formfaible}
\forall v \in C^1_c(\mathbb{R}^d \setminus \overline{\mathcal{O}}), \ \int_{\mathbb{R}^d \setminus \overline{\mathcal{O}}} \nabla \widetilde{w} \cdot \nabla v + \int_{\Gamma_1} \left.\frac{\partial w^{\mathrm{per}}}{\partial n}\right|_{\mathrm{ext}} v - \int_{\mathbb{R}^d \setminus \mathcal{O}} \widetilde{g} v  
= 0, 
\end{equation}
and $\widetilde{w}_{|\partial\mathcal O} = - w_{\mathrm{per}}$ in the trace sense.
\label{defifaible}
\end{definition}

\begin{remarque} 
We could also have written equation \eqref{wtilde2} as a system of PDEs coupled by transmission conditions:
\begin{equation}
\label{system}
\begin{cases}
- \Delta \widetilde{w} = \widetilde{g} \ \ \ \mathrm{in} \ \ \ \mathbb{R}^d \setminus (\overline{\mathcal{O} \cup \mathcal{O}^{\mathrm{per}}}) \\
- \Delta \widetilde{w} =  \widetilde{g} \ \ \ \mathrm{in} \ \ \ \mathcal{O}^{\mathrm{per}} \setminus \overline{\mathcal{O}} \\
\widetilde{w} = - w^{\mathrm{per}} \ \ \mathrm{on} \ \ \Gamma_2 \cup \Gamma_3 \\
 \left.\frac{\partial \widetilde{w}}{\partial n}\right|_{\mathrm{ext}} + \left.\frac{\partial \widetilde{w}}{\partial n}\right|_{\mathrm{int}} = \left.\frac{\partial w^{\mathrm{per}}}{\partial n}\right|_{\mathrm{ext}} \ \mathrm{on} \ \ \Gamma_1
\end{cases}
\end{equation}
The three first equations are obviously necessary. The last equation is necessary to guarantee that $w = w^{\mathrm{per}} + \widetilde{w} \in H^2_{\mathrm{loc}}(\mathbb{R}^d \setminus \overline{\mathcal{O}})$. 
\end{remarque}



Using standard tools of the calculus of variations, one easily proves the following:
\begin{lemme}
\label{lemmin}
Assume that $\widetilde w\in H^1\left(\RR^d\setminus {\mathcal O}\right)$. It is a weak solution of \eqref{wtilde2}
in the sense of Definition~\ref{defifaible}, if and only if it is a solution to the following minimization problem:
\begin{equation}
\label{minPB}
\inf_{\widetilde{w} \in V} 
\left\{ \frac{1}{2} \int_{\mathbb{R}^d \setminus \overline{\mathcal{O}}} |\nabla \widetilde{w}|^2 + \int_{\Gamma_1} \left. \frac{\partial w^{\mathrm{per}}}{\partial n} \right|_{\mathrm{ext}} \widetilde{w} - \int_{\mathbb{R}^d \setminus \overline{\mathcal{O}}} \widetilde{g} \widetilde{w} 
\right\},
\end{equation}
where the minimization space $V$ is defined by
\begin{equation}
\label{min}
V := \left\{ \widetilde{w} \in H^1(\mathbb{R}^d \setminus \overline{\mathcal{O}}) \ \mathrm{s.t} \ \ \widetilde{w}_{|\partial \mathcal{O}} = - w_{\mathrm{per}}  \right\}.
\end{equation}
\end{lemme}

\begin{definition}
  \label{def:extension_V}
  Let $\widetilde w \in V$. We denote by $\widetilde W$ its extension to $\RR^d$ defined by $\widetilde w =
  -w^{\rm per}$ in $\mathcal O$.
\end{definition}

The extension $\widetilde W$ of $\widetilde w$ satisfies $\widetilde{W} \in H^1(\mathbb{R}^d)$ under Assumptions \textbf{(A1)-(A2)} on the sequence $(\mathcal{O}_k)_{k \in \mathbb{Z}^d}$. Figure \ref{figproof4} shows a function $\widetilde{w} \in V$ (extended to $\mathcal{O}$ by $-w^{\mathrm{per}}$). 

In order to study the minimization problem (\ref{minPB}), we will need the following Poincar\'e type inequality on
$V$.
\begin{figure}
\centering
\begin{subfigure}{.5\textwidth}
  \centering
\begin{tikzpicture}[scale=.9]\footnotesize
 \pgfmathsetmacro{\xone}{-1}
 \pgfmathsetmacro{\xtwo}{6}
 \pgfmathsetmacro{\yone}{-1}
 \pgfmathsetmacro{\ytwo}{6}
\begin{scope}<+->;
  \draw[step=5cm,gray,very thin] (\xone,\yone) grid (\xtwo,\ytwo);
\end{scope}

\draw[black] (2.5,3) ellipse (1.5cm and 1.3cm);
\fill[white] (2,2) ellipse (1cm and 1.2cm);
\fill[opacity=0.5, gray!20] (2,2) ellipse (1cm and 1.2cm);
\draw[black] (2,2) ellipse (1cm and 1.2cm);
\draw (2,1.35) node[]{$\mathcal{O}_k^{\mathrm{per}} \setminus \mathcal{O}_k$};
\draw (2.1,2.3) node[]{$\mathcal{O}_k \cap \mathcal{O}_k^{\mathrm{per}}$};
\fill[opacity=0.5,gray!100] (2.5,3) ellipse (1.5cm and 1.3cm);
\draw[black] (2.5,3) ellipse (1.5cm and 1.3cm);
\draw (3.5,3) node[]{\large{$\mathcal{O}_k$}};
\draw[black] (0.9,1.4) node[]{$\Gamma_1^k$};
\draw[black] (3.2,4.35) node[]{$\Gamma_3^k$};
\draw[black] (-0.8,1.8) node[]{$\Gamma_2^k$};
\draw[red] (-0.8,2.2) node[]{$\widetilde{W} = 0$};
\draw[blue] (2,2.8) node[]{$\widetilde{W}=0$};
\draw[blue] (2.5,3.5) node[]{$\widetilde{W}=-w^{\mathrm{per}}$};
\draw[-stealth] (1,4) -- (1.5,3.05);
\draw[red] (0.8,4.2) node[]{$\widetilde{W} = 0$};
\draw[-stealth] (-0.15,2) -- (1.3,2.22);

\end{tikzpicture}
\end{subfigure}%
\begin{subfigure}{.5\textwidth}
  \centering
\begin{tikzpicture}[scale=.9]\footnotesize
 \pgfmathsetmacro{\xone}{-1}
 \pgfmathsetmacro{\xtwo}{6}
 \pgfmathsetmacro{\yone}{-1}
 \pgfmathsetmacro{\ytwo}{6}
\begin{scope}<+->;
  \draw[step=5cm,gray,very thin] (\xone,\yone) grid (\xtwo,\ytwo);
\end{scope}

\fill[gray!20] (2,2) ellipse (0.9cm and 1.2cm);
\draw[black] (2,2) ellipse (0.9cm and 1.2cm);
\draw (2,1.5) node[]{\large{$\mathcal{O}_k^{\mathrm{per}}$}};
\fill[gray!60] (3.7,3.7) ellipse (1cm and 1cm);
\draw[black] (3.7,3.7) ellipse (1cm and 1cm);
\draw[black] (3.7,3.7) ellipse (1cm and 1cm);
\draw (3.7,4) node[]{\large{$\mathcal{O}_k$}};
\draw[blue] (3.8,3.5) node[]{$\widetilde{W} = - w^{\mathrm{per}}$};
\draw[black] (0.8,2) node[]{\large{$\Gamma_1^k$}};
\draw[black] (3.7,2.4) node[]{\large{$\Gamma_3^k$}};

\end{tikzpicture}

\end{subfigure}
\caption{Function $\widetilde{w}$ (its extension $\widetilde{W}$) on a perforated cell with and without overlapping}
\label{figproof4}
\end{figure}
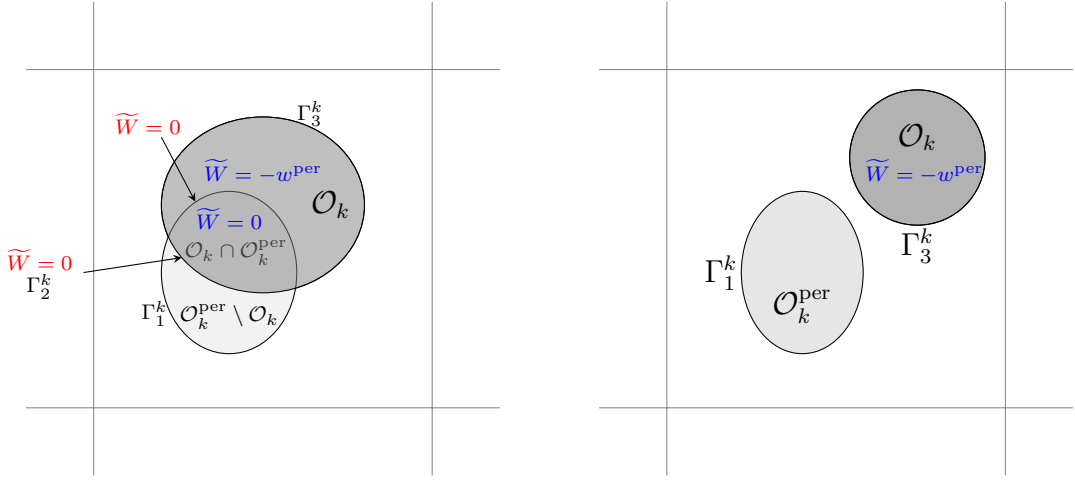

\begin{lemme}[Poincar\'e-type inequality in $V$]
\label{Poincare}
Let $(\mathcal{O}_k)_{k \in \mathbb{Z}^d}$ be a sequence of sets satisfying Assumptions \textbf{(A1)-(A2)}. Define
$\mathcal{O} = \bigcup_{k \in \mathbb{Z}^d} \mathcal{O}_k$. Let $w^{\mathrm{per}}$ be the periodic corrector
solution to \eqref{eq:periodic_corrector}. There exist constants $C_0>0$ and $C_1 > 0$ such that for any $\widetilde{w} \in V$,
\begin{equation}
\label{Poincareineq}
\int_{\mathbb{R}^d \setminus  \overline{\mathcal{O}}} \widetilde{w}^2 \leq C_0 \int_{\mathbb{R}^d \setminus \overline{\mathcal{O}}} | \nabla \widetilde{w} |^2 + C_1.
\end{equation}
Denoting by $\widetilde W$ the extension of $\widetilde{w}$ (see Definition~\ref{def:extension_V}), we also have
$$
\int_{\mathbb{R}^d} \widetilde{W}^2 \leq C_0 \int_{\mathbb{R}^d } | \nabla \widetilde{W} |^2 + C_1.
$$
\end{lemme}

\begin{proof}
Fix $\widetilde{w} \in V$ and extend $\widetilde{w}$ by $-w^{\mathrm{per}}$ in $\mathcal{O}$. This gives a function $\widetilde{W} \in H^1(\mathbb{R}^d)$. Note that 
$$\forall k \in \mathbb{Z}^d, \ \widetilde{W} = 0 \ \mathrm{in} \ \mathcal{O}_k \cap \mathcal{O}_k^{\mathrm{per}}.$$ Fix $k \in \mathbb{Z}^d$, there are two cases :

\textit{Case 1} : $k \in \mathcal{K}$, that is $\mathcal{O}_k \cap \mathcal{O}_k^{\mathrm{per}} = \emptyset$. Then $\widetilde{w} + w^{\mathrm{per}} = 0$ on $\partial\mathcal{O}_k$. Thus classical Poincar\'e inequality gives the existence of $C_k = C(Q_k \setminus \overline{\mathcal{O}_k}, \partial \mathcal{O}_k)$ satisfying $C_k \geq 1$ such that
$$
\int_{Q_k \setminus \overline{\mathcal{O}_k}} (\widetilde{w} + w^{\mathrm{per}})^2 \leq C_k \int_{Q_k \setminus \overline{\mathcal{O}_k}} | \nabla \widetilde{w} + \nabla w^{\mathrm{per}}|^2.
$$
We get
\begin{multline}
\label{BA}
\int_{Q_k \setminus \overline{\mathcal{O}_k}} \widetilde{w}^2 \leq 2 C_k \int_{Q_k \setminus
  \overline{\mathcal{O}_k}} | \nabla \widetilde{w} |^2 + 2 C_k \|w^{\mathrm{per}}\|^2_{H^1(Q_k \setminus
  \overline{\mathcal{O}_k})} \\
 \leq 2 C_k \int_{Q_k \setminus \overline{\mathcal{O}_k}} | \nabla \widetilde{w} |^2 + 2 C_k \|w^{\mathrm{per}}\|^2_{W^{1,\infty}(Q)}.
\end{multline}
Now, the fact that $\widetilde{W} = -w^{\mathrm{per}}$ on $\mathcal{O}_k$ implies
\begin{equation}
\label{AB}
\int_{Q_k} \widetilde{W}^2 \leq 2 C_k \int_{Q_k} | \nabla \widetilde{W} |^2 + 2 C_k \|w^{\mathrm{per}}\|^2_{H^1(Q_k)}\leq 2 C_k \int_{Q_k} | \nabla \widetilde{W} |^2 + 2 C_k \|w^{\mathrm{per}}\|^2_{W^{1,\infty}(Q)}.
\end{equation}

\textit{Case 2 :} $k \notin \mathcal{K}$ so that $\mathcal{O}_k \cap \mathcal{O}^{\mathrm{per}}_k \neq \emptyset$. Note that $\widetilde{W} = 0$ on $\mathcal{O}_k \cap \mathcal{O}^{\mathrm{per}}_k$. We now use Lemma~\ref{H3}: there exists a ball $B_k \subset \mathcal{O}_k \cap \mathcal{O}_k^{\mathrm{per}}$ such $|B_k| \geq \rho$ and thus a box $\mathcal{R}_k \subset \mathcal{O}_k \cap \mathcal{O}_k^{\mathrm{per}}$ such that $|\mathcal{R}_k| \geq C(d)\rho$ where $C(d)$ depends only on the dimension. \\
Theorem \ref{unifpoinc} gives the existence of a constant $C = C(d)/\rho$ chosen $\geq 1$ such that
\begin{equation}
\label{AA}
\int_{Q_k} \widetilde{W}^2 \leq C \int_{Q_k} | \nabla \widetilde{W} |^2.
\end{equation}
Recall that 
$\displaystyle\int_{\mathcal{O}_k} |\nabla \widetilde{W} |^2 \leq \| \nabla w^{\mathrm{per}} \|^2_{L^{\infty}(Q)}  |\mathcal{O}_k \setminus \mathcal{O}_k^{\mathrm{per}} | \leq  \|w^{\mathrm{per}} \|^2_{W^{1,\infty}(Q)} |\mathcal{O}_k \setminus \mathcal{O}_k^{\mathrm{per}}|.$
We thus have
\begin{equation} 
\label{BB}
\int_{Q_k \setminus \overline{\mathcal{O}_k}} \widetilde{w}^2 \leq C \int_{Q_k \setminus \overline{\mathcal{O}_k}} | \nabla \widetilde{w} |^2 + C \|w^{\mathrm{per}} \|^2_{W^{1,\infty}(Q)} |\mathcal{O}_k \setminus \mathcal{O}_k^{\mathrm{per}}|.
\end{equation}

Define
$$
C_0 = \max\left( 2 \max_{k \in \mathcal{K}} C_k, C \right) \ \ \ \ \mathrm{and} \ \ \ \ C_1 = C_0\|w^{\mathrm{per}} \|^2_{W^{1,\infty}(Q)} \left[ |\mathcal{K}| + \sum_{k \in \mathbb{Z}^d} |\mathcal{O}_k \setminus \mathcal{O}_k^{\mathrm{per}} | \right]  < + \infty.
$$
We have proved (see equations \eqref{BA} and \eqref{BB}) that
$$
\forall k \in \mathbb{Z}^d, \ \int_{Q_k \setminus \overline{\mathcal{O}_k}} \widetilde{w}^2 \leq C_0 \int_{Q_k \setminus \overline{\mathcal{O}_k}} |\nabla \widetilde{w}|^2 + C_0 \|w^{\mathrm{per}} \|^2_{W^{1,\infty}(Q)} \delta_k,
$$
where $\delta_k = 1$ if $k \in \mathcal{K}$ and $\delta_k = |\mathcal{O}_k \setminus \mathcal{O}_k^{\mathrm{per}}|$ if $k \notin \mathcal{K}$.

Summing over $k$ gives the desired results for $\widetilde{w}$. Equations \eqref{AB} and \eqref{AA} give the analogous result for $\widetilde{W}$.
\end{proof}

Using Lemma~\ref{Poincare}, we prove the following:
\begin{lemme} 
\label{estim}
Suppose that the sequence $(\mathcal{O}_k)_{k \in \mathbb{Z}^d}$ satisfies Assumption \textbf{(A2)}. Let
$\widetilde{w} \in V$ and denote by $\widetilde W\in H^1(\RR^d)$ its extension (see Definition~\ref{def:extension_V}). Then, one has the following estimates:
\begin{equation}
\label{x}
\left| \int_{\Gamma_1} \left. \frac{\partial w^{\mathrm{per}}}{\partial n} \right|_{\mathrm{ext}} \widetilde{W} \right| \leq C + \frac{1}{4} \|\nabla \widetilde{W} \|_{L^2(\mathbb{R}^d)}^2, 
\end{equation}
where $C$ is a constant independent of $\widetilde{w}$,
\begin{equation}
\label{y}
\left| \int_{\mathbb{R}^d \setminus \mathcal{O}} \widetilde{g} \widetilde{W} \right| \leq \|\widetilde{g} \|_{L^2(\mathbb{R}^d)} \| \widetilde{W} \|_{L^2(\mathbb{R}^d)}
\end{equation}
and
\begin{equation}
\label{z}
\left| \int_{\mathcal{O}^{\mathrm{per}} \setminus \mathcal{O}} \widetilde{W} \right| \leq | \mathcal{O}^{\mathrm{per}} \setminus \mathcal{O}|^{\frac{1}{2}}  \|\widetilde{W} \|_{L^2(\mathbb{R}^d)}.
\end{equation}
\end{lemme}

\begin{proof}
Fix $\widetilde{w} \in V$. Let us first show that $\widetilde{W} \in H^1(\mathbb{R}^d)$. Write
\begin{equation}
\int_{\mathcal{O}} \widetilde{W}^2 + \int_{\mathcal{O}} |\nabla \widetilde{W} |^2  = 
\int_{\mathcal{O} \setminus \mathcal{O}^{\mathrm{per}}} (w^{\mathrm{per}})^2 + \int_{\mathcal{O} \setminus \mathcal{O}^{\mathrm{per}}} |\nabla w^{\mathrm{per}} |^2  \leq \|w^{\mathrm{per}}\|^2_{W^{1,\infty}(Q)} \sum_{k \in \mathbb{Z}^d} |\mathcal{O}_k \setminus \mathcal{O}^{\mathrm{per}}_k|.
\label{Eqintperfo}
\end{equation}
By Lemma \ref{H1}, we conclude that $\widetilde{W} \in H^1(\mathcal{O})$. This proves that $\widetilde{W} \in H^1(\mathbb{R}^d)$.

We now prove estimate \eqref{x}. Standard elliptic regularity implies $\left.\frac{\partial
    w^{\mathrm{per}}}{\partial n}\right|_{\mathrm{ext}} \in L^{\infty}(\partial \mathcal{O}^{\mathrm{per}})$. We
apply the trace theorem \cite[Theorem 1, p 272]{evans} for $p=1$ to the open subset $\mathcal{O}^{\mathrm{per}}_0$ (and thus to $\mathcal{O}_k^{\mathrm{per}}$ by periodicity with the same constant):
\begin{multline}
\label{57}
\left| \int_{\Gamma_1} \left. \frac{\partial w^{\mathrm{per}}}{\partial n} \right|_{\mathrm{ext}} \widetilde{w} \right|  \leq \sum_{k \in \mathbb{Z}^d} \int_{\Gamma_1^k} \left| \left. \frac{\partial w^{\mathrm{per}}}{\partial n} \right|_{\mathrm{ext}} \widetilde{w} \right| 
 \leq \sum_{k \in \mathbb{Z}^d} \int_{\partial \mathcal{O}^{\mathrm{per}}_k} \left| \left. \frac{\partial w^{\mathrm{per}}}{\partial n} \right|_{\mathrm{ext}} \widetilde{W} \right| \\
 \leq  \left\| \left. \frac{\partial w^{\mathrm{per}}}{\partial n} \right|_{\mathrm{ext}} \right\|_{L^{\infty}(\partial \mathcal{O}^{\mathrm{per}}_0)} \sum_{k \in \mathbb{Z}^d} \int_{\partial \mathcal{O}^{\mathrm{per}}_k} |\widetilde{W}| 
\leq C(w^{\mathrm{per}},\mathcal{O}^{\mathrm{per}}) \sum_{k \in \mathbb{Z}^d} \left( \int_{\mathcal{O}^{\mathrm{per}}_k} |\widetilde{W}| + \int_{\mathcal{O}^{\mathrm{per}}_k} | \nabla \widetilde{W} | \right). 
\end{multline} 
Now, recall that $\widetilde{W} = 0$ in $\mathcal{O}_k \cap \mathcal{O}^{\mathrm{per}}_k$, so that using
successively the Cauchy-Schwarz inequality and trace continuity (see \cite[Theorem 1, p 272]{evans} with $p=2$), we have
$$
\left| \int_{\Gamma_1} \left. \frac{\partial w^{\mathrm{per}}}{\partial n} \right|_{\mathrm{ext}} \widetilde{W} \right| \leq C \sum_{k \in \mathbb{Z}^d} |\mathcal{O}_k^{\mathrm{per}} \setminus \mathcal{O}_k |^{1/2} \left(  \|\widetilde{W} \|_{L^2(\mathcal{O}^{\mathrm{per}}_k)} +  \| \nabla \widetilde{W} \|_{L^2(\mathcal{O}^{\mathrm{per}}_k)} \right).  
$$
We use the inequality $ab \leq D \frac{a^2}{2} + \frac{ b^2}{2D}$ with $D$ to be chosen later:
\begin{equation}
\label{58}
\left| \int_{\Gamma_1} \left. \frac{\partial w^{\mathrm{per}}}{\partial n} \right|_{\mathrm{ext}} \widetilde{W} \right| \leq \frac{CD}{2} \sum_{k \in \mathbb{Z}^d} |\mathcal{O}_k^{\mathrm{per}} \setminus \mathcal{O}_k | + \frac{C}{D} \sum_{k \in \mathbb{Z}^d} \left(  \|\widetilde{W} \|^2_{L^2(\mathcal{O}^{\mathrm{per}}_k)} +  \| \nabla \widetilde{W} \|_{L^2(\mathcal{O}^{\mathrm{per}}_k)}^2 \right). 
\end{equation}
Thus,
$$
\left| \int_{\Gamma_1} \left. \frac{\partial w^{\mathrm{per}}}{\partial n} \right|_{\mathrm{ext}} \widetilde{W} \right| \leq \frac{CD}{2} \sum_{k \in \mathbb{Z}^d} |\mathcal{O}_k^{\mathrm{per}} \setminus \mathcal{O}_k | + \frac{C}{D}  \left(  \|\widetilde{W} \|^2_{L^2(\mathbb{R}^d)} +  \| \nabla \widetilde{W} \|_{L^2(\mathbb{R}^d)}^2 \right).
$$
Lemma \ref{Poincare} implies
$$
\frac{C}{D}  \left(  \|\widetilde{W} \|^2_{L^2(\mathbb{R}^d)} +  \| \nabla \widetilde{W} \|_{L^2(\mathbb{R}^d)}^2 \right) \leq \frac{2CC_0}{D}    \| \nabla \widetilde{W} \|_{L^2(\mathbb{R}^d)}^2 + \frac{CC_1}{D}.
$$
Choosing $D = 8CC_0$ yields finally
\begin{equation}
\left| \int_{\partial \mathcal{O}^{\mathrm{per}}} \left. \frac{\partial w^{\mathrm{per}}}{\partial n} \right|_{\mathrm{ext}} \widetilde{W} \right| \leq C \sum_{k \in \mathbb{Z}^d} |\mathcal{O}_k^{\mathrm{per}} \setminus \mathcal{O}_k | + C + \frac{1}{4} \| \nabla \widetilde{W} \|_{L^2(\mathbb{R}^d)}^2, 
\label{2}
\end{equation}
with $C$ being a constant independent of $\widetilde{w}$. We infer \eqref{x} thanks to Lemma \ref{H1}.

The two last estimates \eqref{y} and \eqref{z} are consequences of the Cauchy-Schwarz inequality.
\end{proof}

\begin{remarque}\label{rq:continuite_forme_lineaire}
Let $v \in H^1_0(\mathbb{R}^d \setminus \overline{\mathcal{O}})$. Computations \eqref{57}-\eqref{58} with $\widetilde{w}$ replaced by $v$ and $D=1$ are valid and give
$$
\left| \int_{\Gamma_1} \left. \frac{\partial w^{\mathrm{per}}}{\partial n} \right|_{\mathrm{ext}} v \right| \leq C |\mathcal{O^{\mathrm{per}}} \setminus \mathcal{O} | + C \| v \|_{H^1_0(\mathbb{R}^d \setminus \overline{ \mathcal{O}})}^2.
$$
Thus, the linear form $v \mapsto \int_{\Gamma_1} \left. \frac{\partial w^{\mathrm{per}}}{\partial n} \right|_{\mathrm{ext}} v$ is continuous on $H^1_0(\mathbb{R}^d \setminus \overline{\mathcal{O}})$. 
\end{remarque}

First, we prove below that the minimization space $V$ is not empty: 
\begin{lemme}
Let $(\mathcal{O}_k)_{k \in \mathbb{Z}^d}$ satisfy Assumption \textbf{(A1)} and Assumption \textbf{(A2)}. 
Then $V$ defined by \eqref{min} is not empty.
\label{nonempty}
\end{lemme}

\begin{proof}
We want to build a function $\phi \in H^1(\mathbb{R}^d \setminus \overline{\mathcal{O}})$ satisfying the boundary conditions $\phi = - w^{\mathrm{per}}$ on $\partial \mathcal{O}$. We will first build $\phi$ on each cell $Q_k$. 

Let $k\in \ZZ^d$. Recall that $\delta_0^{\mathrm{per}} = \mathrm{dist}(\mathcal{O}_k^{\mathrm{per}},\partial Q_k)$ and that $\delta_0$ is defined in Lemma \ref{H2} of the Appendix. Set 

$$\varepsilon_k^{\mathrm{per}} = \min(2 \alpha_k, \delta_0^{\mathrm{per}}/2) \ \ \ \mathrm{and} \ \ \ \varepsilon_k = \min(\alpha_k, \delta_0/2)$$
 and note that since $\alpha_k \underset{|k| \rightarrow +\infty}{\longrightarrow} 0$, there exists $k_0$ such that
$$\forall |k| \geq k_0, \ \varepsilon_k^{\mathrm{per}} = 2 \alpha_k \ \ \ \mathrm{and} \ \ \ \varepsilon_k = \alpha_k.$$

Define $\mathcal{U}_k^{\mathrm{per}}(\varepsilon_k^{\mathrm{per}})$ (resp. $\mathcal{U}_k(\varepsilon_k)$) to be the $\varepsilon_k^{\mathrm{per}}$ (resp. $\varepsilon_k$) Minkowski content of $\partial \mathcal{O}_k^{\mathrm{per}}$ (resp. $\partial \mathcal{O}_k$) that is  
$$\mathcal{U}_k^{\mathrm{per}}(\varepsilon_k^{\mathrm{per}}) = \left\{x \in \mathbb{R}^d \ \mathrm{s.t} \ \mathrm{dist}(x,\partial \mathcal{O}_k^{\mathrm{per}}) < \varepsilon_k^{\mathrm{per}} \right\} \subset Q_k$$
and
$$\mathcal{U}_k(\varepsilon_k) = \left\{x \in \mathbb{R}^d \ \mathrm{s.t} \ \mathrm{dist}(x,\partial \mathcal{O}_k) < \varepsilon_k \right\} \subset Q_k.$$ 
Denote
$$\mathcal{O}_k^{\mathrm{per},+}(\varepsilon_k^{\mathrm{per}}) = \mathcal{O}_k^{\mathrm{per}} \cup \mathcal{U}_k^{\mathrm{per}}(\varepsilon_k^{\mathrm{per}}) = \left\{x \in \mathbb{R}^d \ \mathrm{s.t} \ \mathrm{dist}(x, \mathcal{O}_k^{\mathrm{per}}) < \varepsilon_k^{\mathrm{per}} \right\} \subset Q_k$$
and
$$\mathcal{O}_k^{+}(\varepsilon_k) = \mathcal{O}_k \cup \mathcal{U}_k(\varepsilon_k) = \left\{x \in \mathbb{R}^d \ \mathrm{s.t} \ \mathrm{dist}(x, \mathcal{O}_k) < \varepsilon_k \right\} \subset Q_k.$$
Now, let $\chi_k \in C^{\infty}_c(Q_k)$ be a cut-off function satisfying
$$
 \begin{cases}
 0 \leq \chi_k \leq 1 \ \ \mathrm{and} \ \ \chi_k \equiv 1 \ \ \mathrm{in} \ \ \mathcal{O}_k \\
 \mathrm{supp}(\chi_k) \subset \mathcal{O}_k^{+}, \ \ \mathrm{supp}(\nabla \chi_k) \subset \mathcal{U}_k(\varepsilon_k) \\
 | \nabla \chi_k | \leq C / \varepsilon_k.
 \end{cases}
$$

We define $\phi_k = - \chi_k w^{\mathrm{per}}$. It is clear that $\phi_k \in H^1(\mathbb{R}^d)$ and that $\phi_k = - w^{\mathrm{per}}$ on $\partial \mathcal{O}_k$. 

One defines
 $$\phi(x) = \sum_{k \in \mathbb{Z}^d} \phi_k(x) = \sum_{k \in \mathbb{Z}^d} \phi_k(x) 1_{Q_k}(x).$$
 Note that since $\mathrm{supp}(\phi_k) \subset Q_k$, all terms but one (which depends on $x$) vanish in the above sum. Thus $\phi = - w^{\mathrm{per}}$ on $\partial \mathcal{O}$.

 Our goal is to prove that $\phi \in H^1(\mathbb{R}^d \setminus \overline{\mathcal{O}})$ to conclude the proof. By Lemma \ref{Poincare}, it is sufficient to show that $\nabla \phi \in L^2(\mathbb{R}^d \setminus \overline{\mathcal{O}})$. Showing this is equivalent to prove that   
 $$\sum_{k \in \mathbb{Z}^d} \| \nabla \phi_k \|_{L^2(\mathcal{U}_k(\varepsilon_k))}^2 < + \infty.$$
 We are thus left to estimate each term $\| \nabla \phi_k \|_{L^2(\mathcal{U}(\varepsilon_k))}$ where $k \in \mathbb{Z}^d$. We study these terms only when $|k| \geq k_0$ and $k \notin \mathcal{K}$ where $\mathcal{K}$ is defined in Lemma \ref{H1} of the Appendix (there are only a finite number of terms $k$ such that $k \in \mathcal{K}$ and $|k| < k_0$). \\

Let $k \in \mathbb{Z}^d$ such that $|k| \geq k_0$ and $k \notin \mathcal{K}$ that is $\mathcal{O}_k \cap \mathcal{O}_k^{\mathrm{per}} \neq \emptyset$. One has - using Assumption \textbf{(A2)} 
- the inclusions, 
\begin{equation}
\label{inclusionss}
\mathcal{O}_k \subset \mathcal{O}_k^{+}(\alpha_k) \subset \mathcal{O}_k^{\mathrm{per,+}}(2 \alpha_k) \ \ \ \ \mathrm{and} \ \ \ \ \mathcal{U}_k(\alpha_k) \subset \mathcal{U}_k^{\mathrm{per}}(2\alpha_k). 
\end{equation}
We write 
$$
\begin{aligned}
\int_{\mathcal{U}_k(\alpha_k)} \left| \nabla \left(\chi_k w^{\mathrm{per}} \right) \right|^2 & \leq 2 \int_{\mathcal{U}_k(\alpha_k)} | \nabla w^{\mathrm{per}} |^2 |\chi_k|^2 + 2 \int_{\mathcal{U}_k(\alpha_k)} | w^{\mathrm{per}} |^2 | \nabla \chi_k|^2 \\
& \leq 2 \| \nabla w^{\mathrm{per}} \|^2_{L^{\infty}(\mathcal{U}_k(\alpha_k))} |\mathcal{U}_k(\alpha_k)| + 2 \|w^{\mathrm{per}} \|^2_{L^{\infty}(\mathcal{U}_k(\alpha_k))} \frac{C^2}{\alpha_k^2} |\mathcal{U}_k(\alpha_k)|.
\end{aligned}
$$
Using that $\nabla w^{\mathrm{per}}\in L^\infty(\RR^d)$, that
$d\left(\mathcal{U}_k(\alpha_k),\mathcal{O}^{\mathrm{per}}_k\right)\leq \alpha_k$ and that $w^{\mathrm{per}} = 0$ in $\mathcal{O}_k^{\mathrm{per}}$, we infer
$$\| w^{\mathrm{per}} \|_{L^{\infty}(\mathcal{U}_k(\alpha_k))} \leq 2 \alpha_k \| \nabla w^{\mathrm{per}} \|_{L^{\infty}(Q)}.$$
We conclude that 
$$
\int_{\mathcal{U}_k(\alpha_k)} |\nabla \phi_k|^2 \leq C| \mathcal{U}_k(\alpha_k) |  + C |\mathcal{U}_k(\alpha_k) | \alpha_k^2 / \alpha_k^2 \leq C | \mathcal{U}_k(\alpha_k) |.
$$
Using \eqref{inclusionss}, this yields
$$\int_{\mathcal{U}_k(\alpha_k)} |\nabla \phi_k|^2 \leq C | \mathcal{U}_k^{\mathrm{per}}(2\alpha_k)|.$$
We deduce that for $k$ large enough, $\int_{\mathcal{U}_k(\alpha_k)} |\nabla \phi_k|^2 \leq 2 C\alpha_k$ (see \eqref{Mesth}). Since $(\alpha_k)_{k \in \mathbb{Z}^d} \in \ell^1(\mathbb{Z}^d)$, one concludes that $\phi \in H^1(\mathbb{R}^d \setminus \overline{\mathcal{O}})$.
\end{proof}

\begin{proposition}
\label{propmin}
Under the assumptions \textbf{(A1)} and \textbf{(A2)}, the minimization Problem \eqref{minPB} has a solution.
\end{proposition}

\begin{proof}
Let $(\widetilde{w_n})_{n \in \mathbb{N}} \subset V$ be a minimizing sequence of Problem \eqref{minPB} which exists by Lemma \ref{nonempty}, that is
$$
\frac{1}{2} \int_{\mathbb{R}^d \setminus \overline{\mathcal{O}}} |\nabla \widetilde{w_n}|^2 + \int_{\Gamma_1} \left. \frac{\partial w^{\mathrm{per}}}{\partial n} \right|_{\mathrm{ext}} \widetilde{w_n} - \int_{\mathbb{R}^d \setminus \overline{\mathcal{O}}} \widetilde{g} \widetilde{w_n} 
\underset{n \rightarrow +\infty}{\longrightarrow} \inf_{u \in V} J(u).
$$
We extend each $\widetilde{w_n}$ by $- w^{\mathrm{per}}$ in the perforations and denote by $\widetilde{W_n}$ the
extension (see Definition~\ref{def:extension_V}). The sequence
$$
\frac{1}{2} \int_{\mathbb{R}^d} |\nabla \widetilde{W_n}|^2 + \int_{\Gamma_1} \left. \frac{\partial w^{\mathrm{per}}}{\partial n} \right|_{\mathrm{ext}} \widetilde{W_n} - \int_{\mathbb{R}^d \setminus \overline{\mathcal{O}}} \widetilde{g} \widetilde{W_n} 
$$
admits an upper bound independent of $n$.
We first prove that $\| \nabla \widetilde{W_n} \|_{L^2(\mathbb{R}^d)}$ is bounded independently of $n$. We use Lemma \ref{Poincare} and Lemma \ref{estim} to bound each term:
$$\left| \int_{\Gamma_1} \left. \frac{\partial w^{\mathrm{per}}}{\partial n} \right|_{\mathrm{ext}} \widetilde{W_n}
\right| \leq C + \frac{1}{4} \| \nabla \widetilde{W_n} \|^2_{L^2(\mathbb{R}^d)} \ ,$$
$$\left| \int_{\mathbb{R}^d \setminus \overline{\mathcal{O}}} \widetilde{g} \widetilde{W_n} \right| \leq C \|
\widetilde{W_n} \|_{L^2(\mathbb{R}^d)} \underset{\mathrm{Lemma} \ \ref{Poincare}}{\leq} C + C \|\nabla
\widetilde{W_n} \|_{L^2(\mathbb{R}^d)}\ ,$$
where $C$ denotes various constants independent of $n$. Hence, one gets
\begin{displaymath}
C  \geq   \frac{1}{2} \int_{\mathbb{R}^d} |\nabla \widetilde{W_n}|^2 + \int_{\partial \mathcal{O}^{\mathrm{per}}} \left. \frac{\partial w^{\mathrm{per}}}{\partial n} \right|_{\mathrm{ext}} \widetilde{W_n} - \int_{\mathbb{R}^d \setminus \overline{\mathcal{O}}} \widetilde{g} \widetilde{W_n} 
 \geq  \frac{1}{4} \| \nabla \widetilde{W_n} \|^2_{L^2(\mathbb{R}^d)} - C \| \nabla \widetilde{W_n} \|_{L^2(\mathbb{R}^d)} - C,
\end{displaymath}
and thus
$$
\| \nabla \widetilde{W_n} \|^2_{L^2(\mathbb{R}^d)} \leq C \| \nabla \widetilde{W_n} \|_{L^2(\mathbb{R}^d)} + C.
$$
This proves that $\| \nabla \widetilde{W_n} \|_{L^2(\mathbb{R}^d)}$ is bounded independently of $n$. With Lemma \ref{Poincare}, one deduces
that $\| \widetilde{W_n} \|_{H^1(\mathbb{R}^d)}$ is also bounded independently of $n$.

Thus, by weak compactness, there exists a weak limit $\widetilde{W} \in H^1(\mathbb{R}^d)$ such that
$$
\widetilde{W_n} \underset{H^1}{\longrightharpoonup} \widetilde{W} \quad \mathrm{and} \quad \widetilde{W_n} \underset{L^2_{\mathrm{loc}}}{\longrightarrow} \widetilde{W}.
$$

Denote $\widetilde{w} = \widetilde{W}_{|\mathbb{R}^d \setminus \mathcal{O}}$. We first show that $\widetilde{w} \in V$. \\
Strong convergence in $L^2_{\mathrm{loc}}$ and $\widetilde{W}_n = - w^{\mathrm{per}}$ in $\mathcal{O}_k$ imply
$\widetilde{W} = - w^{\mathrm{per}}$ in $\mathcal{O}_k$. For the boundary $\partial \mathcal{O}_k$, recall that the
trace operator $T_k$ (see \cite[Theorem 1, p 272]{evans} ) is weakly continuous from $H^1(\mathcal{O}_k)$ to $L^2(\partial \mathcal{O}_k)$. Thus
$$\left. \widetilde{w} \right|_{\partial \mathcal{O}_k} = T_k \widetilde{W} = - T_k w^{\mathrm{per}} = \left. - w^{\mathrm{per}} \right|_{\partial \mathcal{O}_k}.$$
Since this is true for all $k \in \mathbb{Z}^d$, we have proved that $\widetilde{w}_{| \partial \mathcal{O}} = - w^{\mathrm{per}}$. Moreover, $\widetilde{w} \in H^1(\mathbb{R}^d \setminus \overline{\mathcal{O}})$. Thus $\widetilde{w} \in V$. 

We can now pass to the limit. Since $w \ni H^1(\mathbb{R}^d \setminus \overline{\mathcal{O}}) \mapsto \int |\nabla w |^2$ is convex and continuous (in the strong norm), it is weakly lower semi-continuous and thus
\begin{equation}
\label{11}
 \int_{\mathbb{R}^d \setminus \overline{\mathcal{O}}} |\nabla \widetilde{w}|^2 \leq \liminf_{n \rightarrow +\infty} \int_{\mathbb{R}^d \setminus \overline{\mathcal{O}}} |\nabla \widetilde{w_n}|^2.
\end{equation} 
By weak $H^1-$convergence, since $\widetilde{g} \in L^2(\mathbb{R}^d \setminus \overline{\mathcal{O}})$,
\begin{equation}
\int_{\mathbb{R}^d \setminus \overline{\mathcal{O}}} \widetilde{g} \widetilde{w_n} \underset{n \rightarrow +\infty}{\longrightarrow} \int_{\mathbb{R}^d \setminus \overline{\mathcal{O}} }\widetilde{g} \widetilde{w}.
\label{12}
\end{equation}
 
Let us treat the remaining term. We first recall (see Remark~\ref{rq:continuite_forme_lineaire}) that the linear form 
$v \mapsto \int_{\Gamma_1}  \left. \frac{\partial w^{\mathrm{per}}}{\partial n} \right|_{\mathrm{ext}} v$ is
strongly and thus weakly continous on $H^1_0(\mathbb{R}^d \setminus \overline{\mathcal{O}})$. We apply this
continuity to $v_n = \widetilde w_n - \phi$, where $\phi$ was defined in the proof of Lemma \ref{nonempty}. Since
$\displaystyle \int_{\Gamma_1} \left.\frac{\partial w^{\mathrm{per}}}{\partial n}\right|_{\mathrm{ext}}\phi = 0,$
we deduce
\begin{equation}
\int_{\Gamma_1}  \left. \frac{\partial w^{\mathrm{per}}}{\partial n} \right|_{\mathrm{ext}} \widetilde{w_n} \underset{n \rightarrow +\infty}{\longrightarrow} \int_{\Gamma_1}  \left. \frac{\partial w^{\mathrm{per}}}{\partial n} \right|_{\mathrm{ext}} \widetilde{w}.
\label{13}
\end{equation}

Finally, collecting \eqref{11}, \eqref{12}
and \eqref{13} and letting $n \rightarrow +\infty$, we conclude that 
$$
\frac{1}{2} \int_{\mathbb{R}^d} |\nabla \widetilde{w}|^2 + \int_{\Gamma_1} \left. \frac{\partial w^{\mathrm{per}}}{\partial n} \right|_{\mathrm{ext}} \widetilde{w} - \int_{\mathbb{R}^d \setminus \overline{\mathcal{O}}} \widetilde{g} \widetilde{w} 
\leq \inf_{u \in V} J(u).
$$
This finishes the proof of existence.

\medskip

To conclude the proof, we prove uniqueness: let $\widetilde{w_1}$ and $\widetilde{w_2}$ be two weak solutions of
\eqref{minPB} (in the sense of Definition~\ref{defifaible}). We have that
$$
\forall v \in H^1_0(\mathbb{R}^d \setminus \overline{\mathcal{O}}), \ \ \ \int_{\mathbb{R}^d \setminus \overline{\mathcal{O}}} \nabla \widetilde{w_i} \cdot \nabla v + \int_{\Gamma_1} \left.\frac{\partial w^{\mathrm{per}}}{\partial n}\right|_{\mathrm{ext}} v - \int_{\mathbb{R}^d \setminus \overline{\mathcal{O}}} \widetilde{g} v 
= 0,
$$
for $i=1,2$. Substracting the two equations yields
$$
\forall v \in H^1_0(\mathbb{R}^d \setminus \overline{ \mathcal{O}}), \ \ \ \int_{\mathbb{R}^d \setminus \overline{\mathcal{O}}} \nabla (\widetilde{w_1} - \widetilde{w_2}) \cdot \nabla v = 0
$$

Since $\widetilde{w_1} - \widetilde{w_2} \in H^1_0(\mathbb{R}^d \setminus \overline{\mathcal{O}})$, we may choose $v = \widetilde{w_1}- \widetilde{w_2}$ in the previous expression. 
The Poincar\'e inequality on $Q \setminus \overline{\mathcal{O}}_0$ with $\Gamma = \partial \mathcal{O}_0$ implies $\widetilde{w_1} - \widetilde{w_2} = 0$.
\end{proof}

\begin{remarque}
We could also have applied Lax-Milgram's lemma to show that Problem \eqref{wtilde2} admits a weak solution. The
ingredients are basically the same. Coercivity of the bilinear form is a direct consequence of
Lemma~\ref{unifpoinc} (see \eqref{Po}). Continuity is proved using the same method as in the proof of
Proposition~\ref{propmin}, when passing to the limit in the minimizing sequence. 
\end{remarque}

\subsection{Proof of the convergence results}

\subsubsection{$H^1$ convergence}

\begin{proof}[Proof of Theorem~\ref{theo}]

 We first define the second order approximation of $u_{\varepsilon}$. Let $g = \11_{\mathbb{R}^d \setminus \mathcal{O}}$. With this choice of $g$, one has
$$\widetilde{g} = \11_{\mathcal{O}^{\mathrm{per}} \setminus \mathcal{O} } -  \11_{\mathcal{O} \setminus \mathcal{O}^{\mathrm{per}} } = \sum_{k \in \mathbb{Z}^d} \left( \11_{\mathcal{O}_k^{\mathrm{per}} \setminus \mathcal{O}_k} -  \11_{\mathcal{O}_k \setminus \mathcal{O}^{\mathrm{per}}_k } \right).$$
Moreover, Lemma \ref{H1} implies that $\widetilde{g} \in L^2(\mathbb{R}^d)$. Thus we can apply Theorem \ref{theocor} and get the existence of a unique function $\widetilde{w} \in H^1(\mathbb{R}^d \setminus \overline{\mathcal{O}})$ such that $w := w^{\mathrm{per}} + \widetilde{w}$ satisfies 
$$
\begin{cases}
\begin{aligned}
- \Delta w & = 1 \ \mathrm{in} \ \mathbb{R}^d \setminus \overline{\mathcal{O}} \\
w_{|\partial \mathcal{O}} & = 0.
\end{aligned}
\end{cases}
$$
in sense of distribution. Note that $w \in H^1_{\mathrm{loc}}(\mathbb{R}^d \setminus \mathcal{O})$.

Now, set
$$\phi_{\varepsilon} := u_{\varepsilon} - \varepsilon^2 w(\cdot/\varepsilon)f.$$
Since $f \in \mathcal{D}(\Omega)$, $w = 0$ on $\partial \mathcal{O}$ and $w \in H^1_{\mathrm{loc}}(\mathbb{R}^d  \setminus \mathcal{O})$, one gets that $\phi_{\varepsilon}\in H^1_0(\Omega_{\varepsilon})$.

We have, in the sense of distributions,
\begin{equation}
\label{227}
- \Delta \phi_{\varepsilon}  = f + \Delta w \left(\frac{\cdot}{\varepsilon} \right) f + 2 \varepsilon \nabla w \left( \frac{\cdot}{\varepsilon} \right) \cdot \nabla f + \varepsilon^2 w \left(\frac{\cdot}{\varepsilon}\right)  \Delta f  = f - f + \varepsilon g_{\varepsilon}  = \varepsilon g_{\varepsilon},
\end{equation}
where
$$g_{\varepsilon} =  2 \nabla w \left( \frac{\cdot}{\varepsilon} \right) \cdot \nabla f + \varepsilon w \left(\frac{\cdot}{\varepsilon}\right)  \Delta f.$$
Note that $\| g_{\varepsilon} \|_{L^2(\Omega_{\varepsilon})}$ is bounded independently of $\varepsilon$. 

Next, we multiply \eqref{227} by $\phi_{\varepsilon}$, integrate by parts and apply the Cauchy-Schwarz inequality:
$$\int_{\Omega_{\varepsilon}} |\nabla \phi_{\varepsilon}|^2 = \varepsilon \int_{\Omega_{\varepsilon}} g_{\varepsilon} \phi_{\varepsilon} \leq C \varepsilon \left( \int_{\Omega_{\varepsilon}}  \phi_{\varepsilon}^2 \right)^{1/2}.$$
Thanks to Lemma \ref{Poincmic2}, one concludes that 
$$\left( \int_{\Omega_{\varepsilon}} |\nabla \phi_{\varepsilon}|^2 \right)^{1/2} \leq C \varepsilon^2 \ \ \ \mathrm{and} \ \ \ \left( \int_{\Omega_{\varepsilon}} \phi_{\varepsilon}^2 \right)^{1/2} \leq C \varepsilon^3,$$
which concludes the proof.
\end{proof}

\subsubsection{$L^\infty$ convergence}

We first prove the following Lemma:

\begin{lemme}
\label{lem14}
Let $(\mathcal{O}_k)_{k \in \mathbb{Z}^d}$ be a sequence of open sets satisfying Assumptions
\textbf{(A1)-(A2)}. Let $w$ be the solution to \eqref{wtilde} with $g=1$. 
Then $w \in L^{\infty}\left(\mathbb{R}^d \setminus \overline{\mathcal{O}}\right)$. Moreover, if the $C^{1,\gamma}$ norms of the charts that flatten $\partial \mathcal{O}_k$ are uniformly bounded in $k$, we have that $\nabla w \in L^{\infty}\left(\mathbb{R}^d \setminus \overline{\mathcal{O}}\right)$. 
\end{lemme}

\begin{proof}
Let us first prove that $w \in L^{\infty}(\mathbb{R}^d \setminus \mathcal{O})$. Fix $k \in \mathbb{Z}^d$ and recall that 
\begin{equation}
\label{soussol}
\begin{cases}
\begin{aligned}
- \Delta w & = 1 \ \mathrm{in} \ Q_k \setminus \overline{\mathcal{O}_k} \\
w_{|\partial \mathcal{O}_k} & = 0.
\end{aligned}
\end{cases}
\end{equation}
There exists a constant $C$ independent of $k$ such that 
\begin{equation} 
\label{unifbord}
\| w \|_{L^{\infty}(\partial Q_k)} \leq C.
\end{equation}
Proving \eqref{unifbord} is equivalent to prove that $\| \widetilde{w} \|_{L^{\infty}(\partial Q_k)} \leq C$. Lemma \ref{H2} implies that there exists $\delta > 0$ such that for all $k \in \mathbb{Z}^d$, $\overline{\mathcal{O}_k \cup \mathcal{O}_k^{\mathrm{per}}} \subset [k+\delta,k+1-\delta]^d$. By translation invariance and since $\partial Q$ is compact, there exists $x_1,x_2,...,x_{\ell} \in \partial Q$ such that 
\begin{equation}
\label{boules}
\forall k \in \mathbb{Z}^d, \ \partial Q_k \subset \bigcup_{i=1}^{\ell} B(x_i + k, \delta/2). 
\end{equation}

On each ball $B(x_i + k, \delta)$, $\widetilde{w}$ satisfies $- \Delta \widetilde{w}=  0$. De Giorgi-Nash-Moser Theory (see \cite{zhong}, Theorem 4.22, p. 155) implies that there exists a constant $C = C(d,\delta)$ independent of $x_i$ and $k$ such that
\begin{equation}
\label{Giorgi}
\sup\limits_{ B(x_i + k, \delta/2)} |\widetilde{w}| \leq C(d,\delta) \left( \int_{B(x_i + k,\delta)} |\widetilde{w}(x)|^2 \mathrm{d}x \right)^{\frac{1}{2}} \leq C \| \widetilde{w} \|_{L^2(\mathbb{R}^d \setminus \overline{\mathcal{O}})}.
\end{equation}
The inclusion \eqref{boules} together with \eqref{Giorgi} proves \eqref{unifbord}. We now apply the Maximum
principle on $w$ for each domain $Q_k \setminus \overline{\mathcal{O}_k}$. Let $R$ be such that $Q_k \subset B(k,R)$. The functions
$$w^+(x) := w(x) + \frac{|x-k|^2}{2d} + \| w \|_{L^{\infty}(\partial Q_k)} \ \ \ \mathrm{and} \ \ \ w^-(x) = w(x) + \frac{|x-k|^2 - R^2}{2d} - \| w \|_{L^{\infty}(\partial Q_k)} $$
are respectively supersolution and subsolution of \eqref{soussol}. Thus, thanks to \eqref{unifbord}, $\| w
\|_{L^{\infty}(Q_k \setminus \mathcal{O}_k)}$ is bounded independently of $k$. Hence $w\in L^\infty\left(\RR^d\setminus
\overline{\mathcal O}\right)$.

For $\nabla w$, we use H\"older Regularity results for the first derivatives. First recall that Assumption
\textbf{(A1)} implies that $\mathbb{R}^d \setminus \overline{\mathcal{O}}$ is connected. For all $x \in
\mathbb{R}^d \setminus \overline{\mathcal{O}}$ such that $\mathrm{dist}(x,\partial \mathcal{O})> \delta/2$, there
exists a ball $B_x$ centered at $x$ such that $\mathrm{dist}(B_x,\partial \mathcal{O}) = \delta/2$. Interior
estimates (see \cite{GT}, Theorem 8.32, p. 210) give the existence of a constant $C = C(\delta,d)$ independent of
$x$ such that 
$$\| w \|_{C^{1,\gamma}(B_x)} \leq C \left( \| w \|_{L^{\infty}(\mathbb{R}^d \setminus \mathcal{O})} + 1 \right) \leq C.$$

We have proved that $\nabla w$ is bounded at a distance $\delta/2$ of $\partial \mathcal{O}$.

For the proof up to the boundary $\partial \mathcal{O}$, we use Corollary 8.36 p. 212 of \cite{GT} with the sets $\Omega_k = \{x \ \mathrm{s.t} \ \mathrm{dist}(x,\partial \mathcal{O}_k) < \delta \} \setminus \overline{\mathcal{O}_k}$, $\Omega'_k = \{x \ \mathrm{s.t} \ \mathrm{dist}(x,\partial \mathcal{O}_k) < \delta/2 \} \setminus \overline{\mathcal{O}_k}$ and $T_k = \partial \mathcal{O}_k$. We have $d' = \delta/2$ which is independent of $k$ and thus
$$\| w \|_{C^{1,\gamma}(\Omega'_k)} \leq C(T_k,\delta,d) \left( \| w \|_{L^{\infty}(\mathbb{R}^d \setminus \mathcal{O})} + 1 \right)$$
where the dependence on $T_k$ appears through the $C^{1,\gamma}-$norms of the charts that flatten $T_k$ (see
\cite{GT}, p.210). By hypothesis, we get that $C(T_k) \leq C_0$. This concludes the proof.
\end{proof}

\begin{proof}
[Proof of Theorem \ref{theoinfini}] Fix $\varepsilon > 0$ and define $v_{\varepsilon} = u_{\varepsilon}(\varepsilon \cdot)$. Then $v_{\varepsilon} \in H^1_0(\frac{1}{\varepsilon}\Omega_{\varepsilon})$ and satisfies
\begin{equation}
\label{PBsup}
\begin{cases}
\begin{aligned}
- \Delta v_{\varepsilon} & = \varepsilon^2 f(\varepsilon \cdot) \ \mathrm{in} \ \frac{1}{\varepsilon}\Omega_{\varepsilon} \\
v_{\varepsilon} & = 0 \ \mathrm{on} \ \partial \left( \frac{1}{\varepsilon} \Omega_{\varepsilon} \right).
\end{aligned}
\end{cases}
\end{equation}
Define
$\displaystyle\psi_{\varepsilon} := v_{\varepsilon} - \varepsilon^2 w f(\varepsilon \cdot) \in
H^1_0(\frac{1}{\varepsilon}\Omega_{\varepsilon}),$ and note that 
$$-\Delta \psi_{\varepsilon} = \varepsilon^3 \left[ 2 \nabla w \cdot \nabla f (\varepsilon \cdot) + \varepsilon w \Delta f (\varepsilon \cdot) \right] =: \varepsilon^3 h_{\varepsilon}.$$
Lemma \ref{lem14} and the fact that $f \in \mathcal{D}(\Omega)$ imply that $\|h_{\varepsilon} \|_{L^{\infty}(\frac{1}{\varepsilon} \Omega_{\varepsilon})} \leq C$ for all $0<\varepsilon <1$. Define
$$\psi_{\varepsilon}^{+} = \psi_{\varepsilon} + \varepsilon^3 \|h_{\varepsilon} \|_{L^{\infty}(\frac{1}{\varepsilon} \Omega_{\varepsilon})} \left( w + \|w\|_{L^{\infty}} \right).$$
Then $\psi_{\varepsilon}^{+}$ is a supersolution of Problem \eqref{PBsup}. Thus, by the weak maximum principle (see \cite{GT} Theorem~8.1, p.179), one gets that $\psi_{\varepsilon}^{+} \geq 0$ on $\frac{1}{\varepsilon} \Omega_{\varepsilon}$. Similarly,
$$\psi_{\varepsilon}^{-} = \psi_{\varepsilon} - \varepsilon^3 \|h_{\varepsilon} \|_{L^{\infty}(\frac{1}{\varepsilon} \Omega_{\varepsilon})} \left( w + \|w\|_{L^{\infty}} \right)$$
is a subsolution of \eqref{PBsup} and thus $\psi_{\varepsilon}^{-} \leq 0$ on $\frac{1}{\varepsilon} \Omega_{\varepsilon}$. Finally,
$$- \varepsilon^3 \|h_{\varepsilon} \|_{L^{\infty}(\frac{1}{\varepsilon} \Omega_{\varepsilon})} \left( w + \|w\|_{L^{\infty}} \right) \leq \psi_{\varepsilon} \leq \varepsilon^3 \|h_{\varepsilon} \|_{L^{\infty}(\frac{1}{\varepsilon} \Omega_{\varepsilon})} \left( w + \|w\|_{L^{\infty}} \right).$$ The bound $\|h_{\varepsilon} \|_{L^{\infty}(\frac{1}{\varepsilon} \Omega_{\varepsilon})} \leq C$ and Lemma \ref{lem14} imply $\| \psi_{\varepsilon} \|_{L^{\infty}(\frac{1}{\varepsilon}\Omega_{\varepsilon})} \leq C \varepsilon^3$. Rescaling back concludes the proof.
\end{proof}


\appendix

\section{Proof of technical lemmas}
\label{sec:appendix}

\begin{lemme}
\label{H1}
Let $(\mathcal{O}_k)_{k \in \mathbb{Z}^d}$ be a sequence of open sets satisfying Assumption \textbf{(A2)}. Then,
\begin{equation}
\label{hyp}
\sum_{k \in \mathbb{Z}^d} |\mathcal{O}_k \Delta \mathcal{O}_k^{\mathrm{per}}  | < + \infty,
\end{equation}
where $A \Delta B = (A \cup B) \setminus (A \cap B) = (A \setminus B) \cup (B \setminus A)$ stands for the symmetric subset difference.

Moreover, if
\begin{displaymath}
\mathcal{K} := \{ k \in \mathbb{Z}^d \ \ \mathrm{s.t} \ \ \mathcal{O}_k \cap \mathcal{O}_k^{\mathrm{per}} = \emptyset \},  
\end{displaymath}
then $|\mathcal{K}| < + \infty$.
\end{lemme}

\begin{proof}
First note that, using \eqref{HYP},  
\begin{equation}
\label{lemme}
\mathcal{O}_k \setminus \mathcal{O}_k^{\mathrm{per}} \subset \mathcal{U}_k^{\mathrm{per}}(\alpha_k) \ \ \ \mathrm{and} \ \ \ \mathcal{O}_k^{\mathrm{per}} \setminus \mathcal{O}_k \subset \mathcal{U}_k^{\mathrm{per}}(\alpha_k).
\end{equation}
We now use \cite[Theorem 3.2.39]{federer} to control the measure of $\mathcal{U}_k^{\mathrm{per}}(\alpha_k)$: there exists $\overline{\alpha} > 0$ such that
\begin{equation}
\label{Mesth}
\forall \alpha < \overline{\alpha}, \ \left| \mathcal{U}_0^{\mathrm{per}}(\alpha) \right| \leq 2 C |\partial \mathcal{O}_0^{\mathrm{per}} | \alpha.
\end{equation} 
By translation invariance, the above assertion is true for $\mathcal{U}_0^{\mathrm{per}}(\alpha)$ replaced by $\mathcal{U}_k^{\mathrm{per}}(\alpha)$ : 
$$\forall k \in \mathbb{Z}^d, \ \forall \alpha < \overline{\alpha}, \ \left| \mathcal{U}_k^{\mathrm{per}}(\alpha) \right| \leq 2 C |\partial \mathcal{O}_0^{\mathrm{per}} | \alpha.$$
 For $k$ large enough such that $\alpha_k < \overline{\alpha}$, one thus have
$|\mathcal{U}_k^{\mathrm{per}}(\alpha_k)| \leq \widetilde{C} \alpha_k$ where $\widetilde{C}$ is a constant. This, together with \eqref{lemme}, proves the \eqref{hyp}.

The fact that $|\mathcal{K}| < + \infty$ is a direct consequence of \eqref{hyp} and of the fact that for all $k\in
\mathcal K$,
$\displaystyle \ |\mathcal{O}_k \Delta \mathcal{O}_k^{\mathrm{per}}| \geq |\mathcal{O}_k^{\mathrm{per}} \setminus \mathcal{O}_k| = | \mathcal{O}_0^{\mathrm{per}}|.$
\end{proof}


\begin{lemme}
\label{H3}
Let $(\mathcal{O}_k)_{k \in \mathbb{Z}^d}$ be a sequence of open sets satisfying Assumption \textbf{(A2)}. There exists $\rho > 0$ such that
$$
\forall k \in \mathcal{K}^c, \ \exists B_k \ \mathrm{s.t} \ |B_k| \geq \rho \ \mathrm{and} \ B_k \subset \mathcal{O}_k \cap \mathcal{O}_k^{\mathrm{per}},
$$
where $B_k$ denotes an open ball and $\mathcal{K}$ is defined in Lemma \ref{H1}.
\end{lemme}

\begin{proof}
Since $\mathcal{O}_0^{\mathrm{per}}$ is open, it contains a ball $B \subset \overline{B} \subset \mathcal{O}_0^{\mathrm{per}}$. One has $\delta := \mathrm{dist}(\overline{B},\partial \mathcal{O}_0^{\mathrm{per}}) > 0$.

By translation invariance, for all $k \in \mathbb{Z}^d$, $B_k := B + k$ satisfies 
$$B_k \subset \overline{B_k} \subset \mathcal{O}_k^{\mathrm{per}} \ \ \ \mathrm{and} \ \ \ \delta = \mathrm{dist}(\overline{B_k}, \partial \mathcal{O}_k^{\mathrm{per}}).$$

Since $(\alpha_k)_{k \in \mathbb{Z}^d} \in \ell^1(\mathbb{Z}^d)$, there exists $k_0$ such that for all $|k| \geq k_0$, $\alpha_k \leq \delta/2$. Equation \eqref{Omoins} implies $\overline{B_k} \subset \mathcal{O}_k^{\mathrm{per},-}(\alpha_k)$ for $|k| \geq k_0$. This proves that
 $$ \forall |k| \geq k_0, \ \ \overline{B_k} \subset \mathcal{O}_k^{\mathrm{per}} \cap \mathcal{O}_k.$$

If $|k| < k_0$ and $\mathcal{O}_k^{\mathrm{per}} \cap \mathcal{O}_k \neq \emptyset$, there exists a ball $B_k$ such that $$B_k \subset \mathcal{O}_k^{\mathrm{per}} \cap \mathcal{O}_k.$$

Defining $\displaystyle\rho =\min \left( \min\limits_{|k| < k_0} |B_k|, |B| \right) > 0$ concludes the proof.
\end{proof}

\begin{lemme}
\label{H2}
Let $(\mathcal{O}_k)_{k \in \mathbb{Z}^d}$ be a sequence of open sets satisfying Assumptions \textbf{(A1)-(A2)}. There exists $\delta_0 > 0$ such that $$\forall k \in \mathbb{Z}^d, \ \mathrm{dist}(\mathcal{O}_k, \partial Q_k) \geq \delta_0.$$
\end{lemme}

\begin{proof}
Recall that for all $k \in \mathbb{Z}^d$, $\mathcal{O}_k^{\mathrm{per}} \subset \subset Q_k$. Thus, by translation invariance, there exists a constant $\delta_0^{\mathrm{per}} > 0$ independent of $k$ such that 
$$\forall k \in \mathbb{Z}^d, \ \mathrm{dist}(\mathcal{O}_k^{\mathrm{per}}, \partial Q_k) = \delta_0^{\mathrm{per}}.$$ One has, using Assumption \textbf{(A2)} and in particular the inclusion $\mathcal{O}_k \subset \mathcal{O}^{\mathrm{per},+}_k(\alpha_k)$,
$$\delta_k := \mathrm{dist}(\mathcal{O}_k, \partial Q_k) \geq \delta_0^{\mathrm{per}} - \alpha_k \geq \delta_0^{\mathrm{per}}/2$$ for $k$ large enough, say $|k| \geq k_0$. 

Since for all $|k| < k_0$, Assumption \textbf{(A1)} gives
 $$\delta_k = \mathrm{dist}(\mathcal{O}_k, \partial Q_k) > 0,$$ the Lemma is proved by defining $\delta_0 := \min\left( \frac{\delta_0^{\mathrm{per}}}{2}, \min\limits_{|k| < k_0} \delta_k \right)$. 
\end{proof}

\section*{Acknowledgments}

We thank C. Le Bris for comments and suggestions that greatly improved the manuscript.

  \bibliographystyle{plain}
  \bibliography{article}

\end{document}